\documentclass[12pt]{extarticle}
\usepackage{amsmath, amsthm, amssymb, mathtools, hyperref, color}
\usepackage[shortlabels]{enumitem}
\usepackage{graphicx}
\usepackage{adjustbox}
\usepackage{xspace}
\usepackage{lmodern}
\usepackage{cleveref}
\usepackage{hyperref}
\usepackage{tikz}
\usepackage{multirow}
\usepackage{booktabs}
\usepackage{array}
\usepackage{tabularx}
\usepackage{fancyvrb}
\usetikzlibrary{arrows,calc,fit,matrix,positioning,cd}
\usepackage{diagbox}
\usepackage{comment}
\usepackage{float}
\usepackage{xfrac}

\definecolor{cb-yellow}{RGB}{221,170,51}
\definecolor{cb-red} {RGB}{187,85,102}
\definecolor{cb-green}{RGB}{17,119,51}

\hypersetup{
    colorlinks=true,
    linkcolor=cb-red,
    filecolor=cb-red,      
    urlcolor=cb-green,
    citecolor=cb-green,
}

\usepackage{todonotes}

\oddsidemargin \evensidemargin
\newtheorem{theorem}{Theorem}
\newtheorem*{theorem*}{Theorem}
\numberwithin{theorem}{section}
\newtheorem{proposition}[theorem]{Proposition}
\newtheorem{lemma}[theorem]{Lemma}
\newtheorem{corollary}[theorem]{Corollary}

\newtheorem{definition}{Definition}[section]

\newtheorem{remark}[theorem]{Remark}
\newtheorem{example}[theorem]{Example}

\newcommand{\RR}{\mathbb{R}}

\newcommand{\ZZ}{\mathbb{Z}}
\newcommand{\NN}{\mathbb{N}}

\newcommand{\bmf}[1]{\mathbf{#1}}

\newcolumntype{P}[1]{>{\centering\arraybackslash}p{#1}}
\newcolumntype{C}[1]{>{\centering\arraybackslash}p{#1}}

\usepackage{etoolbox}
\patchcmd{\thebibliography}
  {\settowidth}
  {\setlength{\itemsep}{0pt plus 0.1pt}\settowidth}
  {}{}
\apptocmd{\thebibliography}
  {\small}
  {}{}

\date{}

\title{\textbf{Likelihood Geometry of Moving Average \\ and Autoregressive Processes}}
\author{Carlos Am\'endola \and Gabriel Riffo}

\begin{document}
\maketitle

\begin{abstract}
We study the problem of maximum likelihood estimation for moving average (MA) time series models from the perspective of algebraic statistics, with a focus on the structure and number of solutions of their likelihood equations. Of particular interest is to classify the critical points that lead to non-invertible models. We consider the composite likelihood as an alternative estimation method and analyze its critical points. We extend our algebraic analysis to autoregressive processes (AR). We provide algebraic closed form formulas for the parameters when possible. We also explore in simulations how methods from numerical algebraic geometry perform against traditional optimization for these models. 
\end{abstract}

\section{Introduction}

Moving average (MA) and autoregressive (AR) models play a central role in time series analysis due to their wide range of applications, including modeling financial returns \cite{box2015time}, analyzing climate and environmental data \cite{wilks2011statistical}, and processing engineering signals \cite{OppenheimSchaferBuck1999}. From an algebraic perspective, \cite{autocovariance} introduced the concept of \emph{autocovariance varieties} for MA processes, which represent all possible autocovariance vectors generated by a process of fixed order. This algebraic-geometric viewpoint provides a natural framework for studying parameter identifiability, analyzing the structure of critical points, and enumerating solutions to the maximum likelihood equations. Similarly, AR processes can be analyzed using their autocovariance structure and characteristic polynomials, allowing for a unified algebraic treatment of parameter estimation problems.

\begin{definition}
    Let $q \in \NN$. A \emph{moving average process of order} $q$, denoted by \emph{MA($q$)}, is defined as a sequence of random variables $(Y_t)_{t \in \mathbb{Z}}$ given by
\begin{equation*}
Y_t = \sum_{k=0}^{q} a_k \, Z_{t-k},
\end{equation*}
where $a_k \in \mathbb{R}$ and $(Z_t)_{t \in \mathbb{Z}}$ are i.i.d standard normal random variables. 
\end{definition}

In other words, each observation $Y_t$ is a linear combination of $q+1$ finitely many innovations. The \emph{autocovariance function} of the process, $\gamma: \ZZ \rightarrow \RR$ is
\begin{equation*}
\gamma(h) :=  \mathrm{Cov}(Y_{t+h},Y_t)= \mathrm{Cov}(Y_{h},Y_0)= \begin{cases}
    \sum_{k=0}^{q-|h|} a_k a_{k+h} \quad &\text{if }  0 \leq |h| \le q \\ 
    \qquad  0 \quad  &\text{if }  |h|> q.
\end{cases}
\end{equation*}
Formally, the process is said to be \emph{invertible} if it admits a representation 
as an autoregressive process of infinite order. To characterize invertibility, 
one can associate to the MA($q$) process the characteristic polynomial
\begin{equation}
\Theta(x) = \sum_{k=0}^{q} a_k x^k, \label{eq:charpoly}
\end{equation}
whose roots $z_1,\dots,z_q$ determine the invertibility of the process via the 
spectral factorization identity
\begin{equation*}
\sigma^2 \Theta(x)\Theta(x^{-1}) = \sum_{t \in \mathbb{Z}} \gamma(t)\, x^t.
\end{equation*}
Since $\Theta(x)\Theta(x^{-1})$ is invariant under the substitution 
$z_i \mapsto z_i^{-1}$, any root configuration yields an equivalent autocovariance 
structure. The process is invertible under the $a_i$ parametrization if all roots lie outside the unit 
disk, i.e., $|z_i| > 1$ for all $i$. The process still admits an invertible representation as long as all roots satisfy $|z_i|\neq 1$. Consequently, the process is non-invertible if and only if at least one root satisfies 
$|z_i| = 1$.  

Autoregressive processes are defined as follows.

\begin{definition}\label{Def:Autoregressive}
    Let $p \in \NN$. An \emph{autoregressive process of order} $p$, denoted by \emph{AR($p$)}, is defined as a sequence of random variables $(X_t)_{t \in \mathbb{Z}}$ given by
\begin{equation}
X_t = \sum_{i=1}^{p} \phi_i X_{t-i} + \sigma Z_t,   \label{Eq:autoregresive_process}
\end{equation}
where $\phi_i \in \mathbb{R}$, $\sigma>0$ and $(Z_t)_{t\in \ZZ}$ are i.i.d standard normal random variables. 
\end{definition}
Each $X_t$ depends linearly on its $p$ previous values, and its autocovariance function satisfies the classical \emph{Yule--Walker equations} \cite{brockwell2009time}, which provide a direct link between the parameters $\phi = (\phi_1, \dots, \phi_p)$ and the covariance structure. The process is always causal (invertible), but some parameter choices make the process non-stationary \cite{brockwell2009time}.

In both settings, estimation of the parameters from observed data is a central problem. Let $\mathbf{y} := (y_1, \dots, y_n)^\top$ and $\mathbf{x} := (x_1, \dots, x_n)^\top$ denote finite samples for $t=1,2,\dots,n$ from MA($q$) and AR($p$) processes, respectively. Since the innovations $(Z_t)_{t\in \mathbb{Z}}$ are Gaussian, the random samples $Y=(Y_1,\dots,Y_n)$ and $X=(X_1,\dots,X_n)$ are multivariate Gaussian with covariance matrices $\Sigma(a)$ and $\Sigma(\phi)$ that depend on the model parameters. The corresponding log-likelihood functions are
\begin{align}
\ell_{\text{MA}}(a) = - \log \left(\det\left(\Sigma(a)\right)\right) -  \bmf{y}^\top \Sigma(a)^{-1} \bmf{y} , \label{eq:likMA}\\
\ell_{\text{AR}}(\phi) = - \log \left( \det \left(\Sigma(\phi)\right) \right) - \bmf{x}^\top \Sigma(\phi)^{-1} \bmf{x} , \label{eq:likAR}
\end{align}
and the maximum likelihood estimates (MLEs) are defined as
\begin{align}
\hat{a}_n := \arg \max_{a \in \mathcal{A}} \Big\{  - \log \left(\det\left(\Sigma(a)\right)\right) -  \bmf{y}^\top \Sigma(a)^{-1} \bmf{y} \Big\}, \label{eq:loglikMA} \\
\hat{\phi}_n := \arg \max_{\phi \in \mathcal{F}} \Big\{ - \log \left( \det \left(\Sigma(\phi)\right) \right) - \bmf{x}^\top \Sigma(\phi)^{-1} \bmf{x} \Big\}, \label{eq:loglikAR}
\end{align}
where $\mathcal{A} \subset \mathbb{R}^{q+1}$ and $\mathcal{F} \subset \mathbb{R}^{p}$ are compact parameter spaces. Even for moderate orders $q$ and $p$, the inverse covariance matrices $\Sigma^{-1}$ are typically dense, making the analysis of likelihoods and their critical points both theoretically challenging and computationally demanding. 

The algebraic complexity of maximum likelihood estimation is measured by the \emph{maximum likelihood degree} (ML-degree). 
The \emph{(parametric) maximum likelihood degree} \cite[Definition 2.4]{drton2009lectures} is defined as the number of complex solutions of the likelihood equations obtained by setting the gradient of the log-likelihood with respect to the model parameters equal to zero, for generic data. When the statistical model is described implicitly as an algebraic variety, the \emph{(implicit) maximum likelihood degree} \cite[Corollary 2.2.5]{drton2009lectures} is defined as the number of critical points of the likelihood function restricted to the model, for generic data. 

In this work, we build on the algebraic statistics \cite{sullivant2023algebraic} perspective  of \cite{autocovariance} to study the likelihood problem for MA($q$) models, and we extend this approach to AR($p$) models. The maximum likelihood estimate will typically be a critical point of the maximum likelihood function, so we aim to study the set of all critical points. To compute and analyze them, we use computational algebra tools. These include symbolic computation via Gr\"obner bases and numerical algebraic geometry tools provided by \cite{HomotopyContinuation.jl} and \cite{brysiewicz2025pandora}, which allow for a systematic exploration of all complex critical points in the likelihood landscape.

The paper is organized as follows. In Section~\ref{Sec:MA1}, we analyze the MA($1$) model, proving the conjecture in \cite[Conjecture 5.7]{autocovariance} regarding its maximum likelihood degree (Theorem~\ref{teo:conjetura_amendola}) and providing a complete characterization of its non-invertible solutions (Theorem~\ref{teo:MA1noinvertible}). This extends the partial characterizations presented in \cite[Proposition 5.5 and Example 5.6]{autocovariance} to arbitrary sample sizes, enabling the direct computation of all solutions. This analysis is then extended to the MA($2$) model, where the set of non-invertible solutions exhibit a richer structure. We identify the different types of non-invertible solutions in Theorem~\ref{teo:soluciones_noinv_ma2}. In Section~\ref{Section:maximo_no_invertible}, we study the probability that a non-invertible point corresponds to a maximum, providing both a closed-form expression of the relevant Hessian matrix and Monte-Carlo estimation. In Section~\ref{Section:MA(3)} we extend this analysis to the MA($3$) model to calculate the maximum likelihood degree. Section~\ref{sec:Composite_Likehood} introduces composite likelihood methods, emphasizing their computational advantages and computing the associated algebraic degree. Finally, Section~\ref{sec:autoregresive} extends these techniques to the AR($p$) process, analyzing its likelihood function, and computing its maximum likelihood degree. Throughout, in Section \ref{section:Estimation_MA_process}, Section \ref{sec:Estimation:Composite} and Section \ref{Estimation:AR} we  complement our theoretical results by conducting computational experiments comparing our estimation methodology with a classical optimization approach.  

The code accompanying this paper is available at:
\begin{center}
\url{https://github.com/GabrielRiffoJ/Time_Series_Likehood_Geometry}
\end{center}

\section{Model \texorpdfstring{MA($1$)}{MA(1)}}\label{Sec:MA1}

We begin by addressing the simplest case arising from the general MA($q$) maximum likelihood problem: the MA($1$) process. The autocovariance function of an MA($1$) is known in closed form and vanishes after lag one. Its autocovariance matrix $\Sigma$ takes the simplified form: 

\begin{equation}
\Sigma =
\sigma^2
\begin{pmatrix}
a_1^2+a_0^2 & a_0a_1 & 0 & \cdots & 0 \\
a_0a_1 & a_1^2+a_0^2 & a_0a_1 & \cdots & 0 \\
0 & a_0a_1 & a_1^2+a_0^2 & \cdots & 0 \\
\vdots & \vdots & \vdots & \ddots & \vdots \\
0 & 0 & 0 & \cdots & a_1^2+a_0^2
\end{pmatrix}
\label{eq:Sigma1}
\end{equation}

However, the inverse covariance matrix~$\Sigma^{-1}$ does not admit such a nice representation for arbitrary $n$. In particular, it is no longer sparse.

In Table~\ref{tab:Ma1}, we present the results of evaluating the parametric ML degree for the parameters $(a_0,a_1)$ and the implicit ML degree in $(\gamma_0,\gamma_1)$ coordinates, counting complex critical points of the score equations in each case via symbolic–numeric methods. 

\begin{table}[H]
    \centering
    \begin{tabular}{|c|ccccccccc|}
\hline
Sample Size $n$ & 2 & 3 & 4 & 5 & 6 & 7 & 8 & 9 & 10 \\
\hline
Parametric ML degree & 8 & 16 & 24 & 32 & 40 & 48 & 56 & 64 & 72 \\
Implicit ML degree & 1 & 3 & 5 & 7 & 9 & 11 & 13 & 15 & 17 \\
\hline
\end{tabular}
    \caption{ML degrees of a $\text{MA($1$)}$ process.}
    \label{tab:Ma1}
\end{table}

The table suggests that the parametric ML degree of the MA($1$) model is $8(n-1)$. Indeed, the ML degree of the MA($1$) model was conjectured to be $4(n-1)$ when restricted to solutions $(a_0,a_1)$ where $\mathrm{Re}(a_0)>0$, see \cite[Conjecture 5.7]{autocovariance}.

We will now present a proof of this conjecture. Our argument exploits the Toeplitz structure of the covariance matrix. We point out that such strategy was also used in \cite[Section 3]{anderson1986noninvertible} to study maximum likelihood estimation for MA($1$), and the result could be derived by using the computations therein.

\begin{theorem}\label{teo:conjetura_amendola}{(\cite[Conjecture 5.7]{autocovariance})} 
For $n>1$, the parametric maximum likelihood degree of the \emph{MA$(1)$} model equals $8(n-1)$. 
\end{theorem}

\begin{proof}
Let $\mathbf{y}$ be a sample of size $n$ from an MA($1$) process. Since the matrix $\Sigma$ in \eqref{eq:Sigma1} is a symmetric tridiagonal Toeplitz matrix, its eigenvalues are known \cite{gray2006toeplitz} and given by
\begin{equation*}
\lambda_k = a_0^2 + a_1^2 + 2 a_0 a_1 \cos\left( \frac{k\pi}{n+1} \right),
\qquad k = 1,\dots,n.
\end{equation*}
The corresponding eigenvectors for $1\leq k \leq n$ are
\begin{equation*}
q_k = \left(
\sin\left(\frac{k\pi}{n+1}\right),
\sin\left(\frac{2k\pi}{n+1}\right),
\dots,
\sin\left(\frac{nk\pi}{n+1}\right)
\right)^\top.
\end{equation*}
Let $Q$ be the matrix whose $k$-th column is $q_k$, that is, $Q=(q_1,\dots,q_n)$.
This matrix is orthogonal and does not depend on $a_0$ or $a_1$. Writing $\Sigma = Q D Q^\top$, where $D=\mathrm{diag}(\lambda_1,\dots,\lambda_n)$, the log-likelihood function can be written as
\begin{equation*}
\ell(a_0,a_1)
= \sum_{k=1}^n \left[
\log(\lambda_k) + \frac{z_k^2}{\lambda_k}
\right],
\end{equation*}
where $\bmf{z}=(z_1,\dots,z_n)^\top := Q^\top \bmf{y}$.
Differentiating with respect to $a_0$ and $a_1$, we obtain the likelihood equations:
\begin{align*}
\sum_{k=1}^{n}
\left( 2a_0 + 2a_1 \cos\left( \frac{k\pi}{n+1} \right) \right)
\frac{\lambda_k - z_k^2}{\lambda_k^2} &= 0,\\
\sum_{k=1}^{n}
\left( 2a_1 + 2a_0 \cos\left( \frac{k\pi}{n+1} \right) \right)
\frac{\lambda_k - z_k^2}{\lambda_k^2} &= 0.
\end{align*}
To simplify the system, we can set $b= \frac{a_1}{a_0}$, since $a_0 \neq 0$. Substituting yields
\begin{align*}
\sum_{k=1}^{n}
\left( 1 + b \cos\left( \frac{k\pi}{n+1} \right) \right)
\dfrac{a_0^2\left(b^2+1+2b\cos\left( \frac{k\pi}{n+1} \right)\right) - z_k^2}
{\left( b^2+1+2b\cos\left( \frac{k\pi}{n+1} \right) \right)^2} &= 0,\\
\sum_{k=1}^{n}
\left( b + \cos\left( \frac{k\pi}{n+1} \right) \right)
\dfrac{a_0^2\left(b^2+1+2b\cos\left( \frac{k\pi}{n+1} \right)\right) - z_k^2}
{\left( b^2+1+2b\cos\left( \frac{k\pi}{n+1} \right) \right)^2} &= 0.
\end{align*}
Multiplying the second equation by $b$ and adding both equations, we obtain
\begin{equation*}
\sum_{k=1}^{n}
\dfrac{a_0^2\left(b^2+1+2b\cos\left( \frac{k\pi}{n+1} \right)\right) - z_k^2}
{b^2+1+2b\cos\left( \frac{k\pi}{n+1} \right)} = 0.
\end{equation*}
Solving for $a_0^2$ gives
\begin{equation}
a_0^2 =
\frac{1}{n}\sum_{k=1}^{n}
\dfrac{z_k^2}{b^2+1+2b\cos\left(\frac{k\pi}{n+1}\right)}.\label{eq:a0^2} 
\end{equation}
Evaluating $a_0^2$ in the first equation, we obtain
\begin{align}
\dfrac{1}{n}\sum_{j=1}^{n} \sum_{k=1}^{n} 
& \left( 1 + b \cos\left( \frac{k\pi}{n+1} \right) \right)
\dfrac{z_j^2}{
\left( b^2+1+2b\cos\left( \frac{k\pi}{n+1} \right) \right)
\left(b^2+1+2b\cos\left(\frac{j\pi}{n+1}\right)\right)
} \notag \\
&= \sum_{k=1}^{n} 
\left( 1 + b \cos\left( \frac{k\pi}{n+1} \right) \right)
\dfrac{z_k^2}{\left( b^2+1+2b\cos\left( \frac{k\pi}{n+1} \right) \right)^2}.
\label{Equation:polynomial}
\end{align}
Multiplying the equation by 
$\prod_{k=1}^{n} \left( b^2 + 1 + 2b \cos\!\left( \frac{k\pi}{n+1} \right) \right)^2$
yields a polynomial of degree $4(n-1)+1$ in $b$. Since $b=0$ is a solution of \eqref{Equation:polynomial} but does not correspond to an MA$(1)$ process, there are generically $4(n-1)$ possible values for $b$. Substituting into \eqref{eq:a0^2}, we obtain that for every $b$, there are 2 possible solutions for $a_0$, which in turn determines $a_1$ through $a_1= b \cdot a_0$, so that in total there are $2 \cdot 4(n-1) = 8(n-1)$ solutions for $(a_0,a_1)$. 
since and hence at most $8(n-1)+2$ solutions. 
 admissible solutions for the MA$(1)$ model, which proves the claim.
\end{proof}

\begin{remark}
For identifiability of the \emph{MA$(1)$} model, we could retain solutions in Equation~\eqref{eq:a0^2} for which the real part of $a_0$ is positive. Under this restriction, we obtain the $4(n-1)$ solutions of \cite[Conjecture 5.7]{autocovariance}.   
\end{remark}

Table \ref{tab:Ma1} also clearly suggests that the maximum likelihood degree of MA$(1)$  with $n$ samples equals $2n-3$ when considered implicitly in $(\gamma_0,\gamma_1)$-coordinates. Before proving this statement, we turn to the important concept of non-invertibility.  

In the MA($1$) case, the characteristic polynomial in \eqref{eq:charpoly} takes the form $\Theta(x)=a_0+a_1x$, and thus the unique root $x=-a_0/a_1$ lies in the unit circle precisely when $\left|a_0\right |=\left|a_1\right| $, which corresponds to a boundary configuration where the model fails to admit an invertible representation. This situation is particularly delicate, as the usual asymptotic distribution of the estimators breaks down at $\left|a_0\right| = \left|a_1\right|$, rendering standard inference invalid \cite{plosser1977estimation}. Moreover, it has been observed that the maximum likelihood estimator (MLE) for the MA($1$) model may lie on this boundary and, contrary to earlier beliefs, this happens with positive probability \cite{cryer1981small,anderson1986noninvertible}.

In \cite[Proposition 5.5 and Example 5.6]{autocovariance}, non-invertible solutions were characterized only for $n \leq 3$. Here, we extend this analysis to all sample sizes $n$. We first make the following observation regarding the inverses of symmetric tridiagonal Toeplitz matrices. While a general formula for the inverse of tridiagonal matrices can be expressed 
in terms of Chebyshev polynomials via a recursive procedure 
\cite[Corollary 4.1]{da2001explicit}, we derive here a closed-form expression for two particular matrices. 

\begin{lemma}\label{Lemma:inverse_toeplitz}
Let $S^+$ and $S^-$ be the $n\times n$ matrices 
\begin{equation}
S^+ =
\begin{bmatrix} 
2 & 1 & 0 & \cdots & 0 \\
1 & 2 & 1 & \ddots & \vdots \\
0 & 1 & 2 & \ddots & 0 \\
\vdots & \ddots & \ddots & \ddots & 1 \\
0 & \cdots & 0 & 1 & 2 
\end{bmatrix}, \qquad S^- =
\begin{bmatrix} 
2 & -1 & 0 & \cdots & 0 \\
-1 & 2 & -1 & \ddots & \vdots \\
0 & -1 & 2 & \ddots & 0 \\
\vdots & \ddots & \ddots & \ddots & -1 \\
0 & \cdots & 0 & -1 & 2 
\end{bmatrix}. \label{eq:matrixS}
\end{equation}
The (lower) entries of the symmetric matrices $(S^+)^{-1}$ and $(S^-)^{-1}$ are given explicitly by
\begin{equation}
(S^+)^{-1}_{ij} = (-1)^{i+j} \frac{j(n+1-i)}{n+1}, \quad 
(S^-)^{-1}_{ij} = \frac{j(n+1-i)}{n+1}, \quad i \ge j. \label{eq:invS}
\end{equation}
\end{lemma}

\begin{proof}
We prove the first equality for  $(S^+)^{-1}$; the case of $S^-$ is analogous.
Let $Q$ be the matrix defined by the right hand side values in the first equation of \eqref{eq:invS}. We verify that $S^+ Q = I$. Since $S^+$ is tridiagonal with 
diagonal entries $2$ and off-diagonal entries $1$, we have
\[
(S^+ Q)_{ij} = Q_{i-1,j} + 2Q_{i,j} + Q_{i+1,j},
\]
where boundary terms are omitted when $i=1$ or $i=n$. We consider two cases.

\textit{Case 1: $i > j$.} We have
\begin{align*}
(S^+ Q)_{ij} 
&= (-1)^{i-1+j}\frac{j(n+1-i+1)}{n+1} + 2(-1)^{i+j}\frac{j(n+1-i)}{n+1} 
+ (-1)^{i+1+j}\frac{j(n+1-i-1)}{n+1} \\
&= \frac{(-1)^{i+j}\,j}{n+1}\bigl[-(n+2-i) + 2(n+1-i) - (n-i)\bigr] = 0.
\end{align*}

\textit{Case 2: $i = j$.} Using symmetry $Q_{i-1,i} = Q_{i,i-1}$,
\begin{align*}
(S^+ Q)_{ii} 
&= Q_{i-1,i} + 2Q_{i,i} + Q_{i+1,i} \\
&= (-1)^{2i-1}\frac{(i-1)(n+1-i)}{n+1} + 2(-1)^{2i}\frac{i(n+1-i)}{n+1} 
+ (-1)^{2i+1}\frac{i(n-i)}{n+1} \\
&= \frac{1}{n+1}\bigl[-(i-1)(n+1-i) + 2i(n+1-i) - i(n+1-i-1)\bigr] \\
&= \frac{1}{n+1}\bigl[(n+1-i) + i\bigr] = 1.
\end{align*}
Hence $S^+Q = I$ and $(S^+)^{-1}=Q$, as claimed.
\end{proof}

\begin{theorem}[Non-invertible solutions of the MA($1$) process]\label{teo:MA1noinvertible}
Let $\mathbf{y}=(y_1, \dots, y_n)^{\top}$ be a sample from a \emph{MA($1$)} process. The non-invertible critical points of the likelihood, corresponding to the cases $|a_0| = |a_1|$, are explicitly given as follows:

\begin{enumerate}
    \item For $a_0 = a_1$, the solutions are
    \begin{equation*}
        a_0 = a_1 = \pm \sqrt{ \sum_{k=1}^n \frac{k(n+1-k)}{n+1} y_k^2 + \sum_{i=2}^{n} \sum_{j=1}^{i-1} (-1)^{i+j} \frac{j(n+1-i)}{n+1} y_i y_j }.
    \end{equation*}
    
    \item For $a_0 = -a_1$, the solutions are
    \begin{equation*}
        a_0 = -a_1 = \pm \sqrt{ \sum_{k=1}^n \frac{k(n+1-k)}{n+1} y_k^2 + \sum_{i=2}^{n} \sum_{j=1}^{i-1} \frac{j(n+1-i)}{n+1} y_i y_j }.
    \end{equation*}
\end{enumerate}
\end{theorem}

\begin{proof}
Suppose $a_0=a_1$ (the case $a_0 = -a_1$ is analogous), and consider the tridiagonal matrices $S^{+}$ from \eqref{eq:matrixS}. Then we can express the autocovariance matrix $\Sigma$ from \eqref{eq:Sigma1} as
\begin{equation*}
\Sigma = a_0^2 S^{+}, \quad \text{so that} \quad \frac{\partial \Sigma}{\partial a_0} = 2 a_0 S^{+}.
\end{equation*}

The derivative of the log-likelihood \eqref{eq:loglikMA} with respect to $a_0$ is
\begin{align*}
\frac{\partial \ell}{\partial a_0} &= \operatorname{tr}\Big(\Sigma^{-1} \frac{\partial \Sigma}{\partial a_0}\Big) - \bmf{y}^\top \Sigma^{-1} \frac{\partial \Sigma}{\partial a_0} \Sigma^{-1} \bmf{y} \\
&= \operatorname{tr}\Big(\frac{1}{a_0^2} (S^{+})^{-1} (2 a_0 S^{+})\Big) - \bmf{y}^\top \frac{1}{a_0^2} (S^{+})^{-1} (2 a_0 S^{+}) \frac{1}{a_0^2} (S^{+})^{-1} \bmf{y} \\
&= \frac{2n}{a_0} - \frac{2}{a_0^3} \bmf{y}^\top (S^{+})^{-1} \bmf{y}.
\end{align*}
so that equating to $0$ and solving for $a_0^2$ gives
\begin{equation*}
a_0^2 = \frac{1}{n} \bmf{y}^\top (S^{+})^{-1} \bmf{y}.
\end{equation*}
Substituting the values for $(S^+)^{-1}$ from \eqref{eq:invS}, we get
\begin{equation*}
a_0^2 = \sum_{k=1}^n \frac{k(n+1-k)}{n+1} y_k^2 + \sum_{i=2}^n \sum_{j=1}^{i-1} (-1)^{i+j} \frac{j(n+1-i)}{n+1} y_i y_j.
\end{equation*}
The Cholesky decomposition of $S^{+}$ is
\begin{equation*}
L =
\begin{bmatrix} 
\sqrt{2} & 0 & 0 & \cdots & 0 \\
\frac{1}{\sqrt{2}} & \sqrt{\frac{3}{2}} & 0 & \ddots & \vdots \\
0 & \sqrt{\frac{2}{3}} & \sqrt{\frac{3}{4}} & \ddots & 0 \\
\vdots & \ddots & \ddots & \ddots & 0 \\
0 & \cdots & 0 & \sqrt{\frac{n-1}{n}} & \sqrt{\frac{n+1}{n}} 
\end{bmatrix}.
\end{equation*}
This can be proved by induction: Let $L_n$ the matrix $L$ of size $n\times n$.
For $n=1$ it holds. Assuming it holds for $n-1$, for the $n$-th row we have

$$(L_n)_{n,n-1} = \sqrt{\frac{n-1}{n}}, \qquad
(L_n)_{n,n} = \sqrt{2 - \frac{n-1}{n}} = \sqrt{\frac{n+1}{n}}.$$

Thus, the induction is complete.
Since the matrix $S^{+}$ has a Cholesky decomposition, then $a_0^2>0$ and the two solutions are real numbers.
\end{proof}

\begin{remark}
The previous proof relies on an explicit closed-form expression for the inverse 
of the covariance matrix, which becomes substantially more complex for $q > 1$. 
\end{remark}

Building on our identification of the invertible and non-invertible solutions, we can now clarify the relationship between solutions expressed in the autocovariances $\gamma_i$ and those expressed in the moving-average parameters $(a_0,a_1)$ of MA($1$). 

\begin{corollary}
    The implicit maximum likelihood degree of \emph{MA($1$)} is $2n-3$ in $(\gamma_0,\gamma_1)$. For each solution in the $\gamma_i$-coordinates, the map $\gamma(a_0,a_1) \longmapsto (\gamma_0,\gamma_1)$ produces 4 preimages, in addition to the four non-invertible solutions corresponding to the boundary case $\left|a_0\right|=\left|a_1\right|$. 
\end{corollary}
\begin{proof}
The fact that  $\gamma(a_0,a_1)$ admits four distinct preimages is proved in \cite[Proposition 4.1]{autocovariance}.
     Specifically, the system formulated in the $\gamma_i$-coordinates has $2n-3$ solutions, while the lifted system in the $(a_0,a_1)$-coordinates has 
$4(2n-3)+4 = 8(n-1)$ 
solutions in total, precisely matching the predicted maximum in Theorem~\ref{teo:conjetura_amendola}. 
\end{proof}
\begin{remark}\label{rmk:bound1}
    In \cite[Proposition 5.1]{Sturmfels2019} and \cite[Theorem 2.2]{Coons2019}, it is shown that generic two-dimensional Gaussian covariance models have an implicit ML degree of $2n-3$. Our results show that the  \emph{MA($1$)} model is generic enough in the sense that this upper bound is attained. This phenomenon is specific to \emph{MA($1$)}. Indeed, already for \emph{MA($2$)} the number of solutions is less than the generic upper bound, see Remark~\ref{rmk:bound2}. This reflects the more complex algebraic structure introduced by the additional parameters in higher order models.
\end{remark}

\section{Model \texorpdfstring{$MA(2)$}{MA(2)}}
  
The MA($2$) model introduces additional complexity compared to MA($1$),
as the autocovariance matrix $\Sigma(\gamma)$ is now pentadiagonal Toeplitz, and no closed-form expression for its inverse exists in general. More explicitly, the autocovariance matrix has the form:
\begin{equation}
\Sigma(\gamma) = 
\begin{bmatrix}
\gamma_0 & \gamma_1 & \gamma_2 & 0 & \cdots & 0  & 0  & 0\\
\gamma_1 & \gamma_0 & \gamma_1 & \gamma_2 & \cdots & 0  & 0  & 0  & \\
\gamma_2 & \gamma_1 & \gamma_0 & \gamma_1 & \cdots & 0 & 0  & 0  \\
0 & \gamma_2 & \gamma_1 & \gamma_0 & \cdots & 0 & 0  & 0  \\
\vdots & \vdots & \vdots & \vdots & \ddots & \vdots & \vdots  & \vdots\\
0 & 0 & 0 & 0 & \cdots & \gamma_0 & \gamma_1  & \gamma_2  \\ 
0 & 0 & 0 & 0 & \cdots & \gamma_1 & \gamma_0  & \gamma_1 \\
0 & 0 & 0 & 0 & \cdots & \gamma_2 & \gamma_1  & \gamma_0  
\end{bmatrix} ,  \label{Eq:MA2_Sigma_gamma}
\end{equation}
with the respective derivatives:
\begin{small}
\begin{equation}
\frac{\partial \Sigma}{\partial \gamma_0}(\gamma) = 
\begin{bmatrix}
1 & 0& 0 & 0 & \cdots & 0  & 0  & 0\\
0 & 1 & 0 & 0 & \cdots & 0  & 0  & 0  & \\
0 & 0 & 1 & 0 & \cdots & 0 & 0  & 0  \\
0 & 0 & 0 & 1 & \cdots & 0 & 0  & 0  \\
\vdots & \vdots & \vdots & \vdots & \ddots & \vdots & \vdots  & \vdots\\
0 & 0 & 0 & 0 & \cdots & 1 & 0  & 0  \\ 
0 & 0 & 0 & 0 & \cdots & 0 & 1  & 0 \\
0 & 0 & 0 & 0 & \cdots & 0 & 0  & 1  
\end{bmatrix}, 
\quad \frac{\partial \Sigma}{\partial  \gamma_1}(\gamma) = 
\begin{bmatrix}
0 & 1 & 0 & 0 & \cdots & 0  & 0  & 0\\
1 & 0 & 1 & 0 & \cdots & 0  & 0  & 0  & \\
0 & 1 & 0 & 1 & \cdots & 0 & 0  & 0  \\
0 & 0 & 1 & 0 & \cdots & 0 & 0  & 0  \\
\vdots & \vdots & \vdots & \vdots & \ddots & \vdots & \vdots  & \vdots\\
0 & 0 & 0 & 0 & \cdots & 0 & 1  & 0 \\ 
0 & 0 & 0 & 0 & \cdots & 1 & 0  & 1 \\
0 & 0 & 0 & 0 & \cdots & 0 & 1  & 0  
\end{bmatrix}
\label{Eq:MA2_Sigma_derivate}
\end{equation}
\begin{equation}
 \frac{\partial \Sigma}{\partial \gamma_2}(\gamma) = 
\begin{bmatrix}
0 & 0 & 1 & 0 & \cdots & 0  & 0  & 0\\
0 & 0 & 0 & 1 & \cdots & 0  & 0  & 0  & \\
1 & 0 & 0 & 0 & \cdots & 0 & 0  & 0  \\
0 & 1 & 0 & 0 & \cdots & 0 & 0  & 0  \\
\vdots & \vdots & \vdots & \vdots & \ddots & \vdots & \vdots  & \vdots\\
0 & 0 & 0 & 0 & \cdots & 0 & 0  & 1  \\ 
0 & 0 & 0 & 0 & \cdots & 0& 0  & 0 \\
0 & 0 & 0 & 0 & \cdots & 1 & 0 & 0  
\end{bmatrix}. \label{Eq:MA2_Sigma_derivate2}
\end{equation}
\end{small}
Substituting the parametrization in $a$, we obtain the autocovariance matrix:
\begin{footnotesize}
\begin{equation}
\Sigma(a) = 
\begin{bmatrix}
a_0^2+a_1^2+a^2 & a_1(a_0+a_2) & a_0a_2 & 0 & \cdots & 0 \\
a_1(a_0+a_2) & a_0^2+a_1^2+a_2^2 & a_1(a_0+a_2) & a_0a_2 & \cdots & 0 \\
a_0a_2 & a_1(a_0+a_2) & a_0^2+a_1^2+a_2^2 & a_1(a_0+a_2) & \cdots & 0 \\
0 & a_0 a_2 & a_1(a_0+a_2) & a_0^2+a_1^2+a_2^2 & \cdots & 0 \\
\vdots & \vdots & \vdots & \vdots & \ddots & \vdots \\
0 & 0 & 0 & 0 & \cdots & a_0^2+a_1^2+a_2^2
\end{bmatrix}
\label{Eq:MA2_Sigma_a}
\end{equation}
\end{footnotesize}

and the corresponding derivatives:
\begin{footnotesize}
\setlength{\arraycolsep}{3pt}
\begin{equation}
   \frac{\partial \Sigma}{ \partial a_0}(a)=  
\begin{bmatrix}
2a_0 & a_1 & a_2 & 0 & \cdots & 0 \\
a_1 & 2a_0 & a_1 & a_2 & \cdots & 0 \\
a_2 & a_1 & 2a_0 & a_1 & \cdots & 0 \\
0 & a_2 & a_1 & 2a_0 & \cdots & 0 \\
\vdots & \vdots & \vdots & \vdots & \ddots & \vdots \\
0 & 0 & 0 & 0 & 0 & 2a_0
\end{bmatrix}
, \;  \frac{\partial \Sigma}{ \partial a_1}(a)= 
\begin{bmatrix}
2a_1 & a_0+a_2 & 0 & 0 & \cdots & 0 \\
a_0+a_2 & 2a_1 & a_0+a_2 & 0 & \cdots & 0 \\
0 & a_0+a_2 & 2a_1 & a_0+a_2 & \cdots & 0 \\
0 & 0 & a_0+a_2 & 2a_1 & \cdots & 0 \\
\vdots & \vdots & \vdots & \vdots & \ddots & \vdots \\
0 & 0 & 0 & 0 & \cdots & 2a_1
\end{bmatrix}
 \label{Eq:MA2_Sigma_a_derivate} 
\end{equation}
\end{footnotesize}
    \begin{footnotesize}
\begin{equation}
   \frac{\partial \Sigma}{ \partial a_2}(a)= 
\begin{bmatrix}
2a_2 & a_1 & a_0 & 0 & \cdots & 0 \\
a_1 & 2a_2 & a_1 & a_2 & \cdots & 0 \\
a_0 & a_1 & 2a_2 & a_1 & \cdots & 0 \\
0 & a_0 & a_1 & 2a_2 & \cdots & 0 \\
\vdots & \vdots & \vdots & \vdots & \ddots & \vdots  \\
0 & 0 & 0 & 0 & \cdots & 2a_2
\end{bmatrix}
\label{Eq:MA2_Sigma_a_ derivate2}
\end{equation}
\end{footnotesize}

\subsection{Maximum Likelihood Degrees}
As a consequence of not having access to an explicit inverse of \eqref{Eq:MA2_Sigma_gamma}, analytical derivations of the ML degree become infeasible for a general sample size $n$, and we proceed using computational methods. In \cite{autocovariance}, the authors used Gröbner bases to compute the solutions of the score equations symbolically, obtaining the parametric maximum likelihood degree up to the case $n = 6$. Since Gröbner bases computations can become extremely computationally demanding as the sample size increases, to extend this analysis beyond such limitations, we use the \texttt{HomotopyContinuation.jl} package~\cite{HomotopyContinuation.jl} in \texttt{Julia}, which allows us to efficiently track all complex solutions of the polynomial system defined by the score equations. Additionally, we used the \texttt{Pandora} package~\cite{brysiewicz2025pandora} to visualize the number of solutions. 
We compute the parametric ML degree using $(a_0,a_1,a_2)$ and the implicit ML degree using the autocovariances $(\gamma_0, \gamma_1, \gamma_2)$. The results are summarized in Table~\ref{tab:MA(2)_gamma_res}. 

\begin{table}[H]
    \centering
    \begin{tabular}{|c|cccccccc|}
\hline
Sample Size $n$ & 3 & 4 & 5 & 6 & 7 & 8 & 9 & 10 \\
\hline
Parametric ML degree & 58 & 138 & 258 & 410 & 578 & 778 & 954 & 1136 \\
Implicit ML degree & 2 & 9 & 21 & 37 & 55 & 77 & 96 & 116 \\
\hline
\end{tabular}
    \caption{ML degrees of a $\text{MA($2$)}$ process.}
    \label{tab:MA(2)_gamma_res}
\end{table}

Figure~\ref{fig:MA2_real_a} and Figure~\ref{fig:MA2_real_gamma} show the real solution maps for selected sample sizes, for the parametric and implicit ML degree. These maps were generated using the \texttt{Pandora.jl} package~\cite{brysiewicz2025pandora}, which allows exploring how the number of real solutions changes as a single sample value $y$ is varied while keeping all other sample values fixed. By systematically changing $y$, Pandora evaluates the number of real solutions for each configuration, providing a visualization of the distribution of solutions in the parameter 
space. 

The solution maps reveal a notable difference between the two formulations. In the implicit case (Figure~\ref{fig:MA2_real_gamma} ), the number of real solutions reaches the maximum possible value (equal to the implicit ML degree) for a significant portion of the sampled configurations, whereas in the parametric case (Figure~\ref{fig:MA2_real_a}) this maximum is never attained. This suggests a structural distinction: we conjecture that for the implicit formulation the maximum number of real solutions is always achievable for generic sample values, while for the parametric formulation it is not.

Figure~\ref{fig:MA2_n3_solution} shows the solution maps for $n=3$, indicating which 
of the two solutions---positive (blue) or negative (red)---is selected for each 
parameter configuration. Since a full three-dimensional visualization is not feasible, 
one variable $y_i$ is fixed in each panel while the remaining two are varied. A key 
observation is that the boundaries separating both regions are non-trivial: even after 
fixing one variable, neither solution dominates over a simple subdomain, reflecting a 
complex dependence of the solution selection on the parameters.

\begin{figure}[H]
    \centering
    \begin{minipage}{0.32\linewidth}
        \centering
        \includegraphics[width=\linewidth]{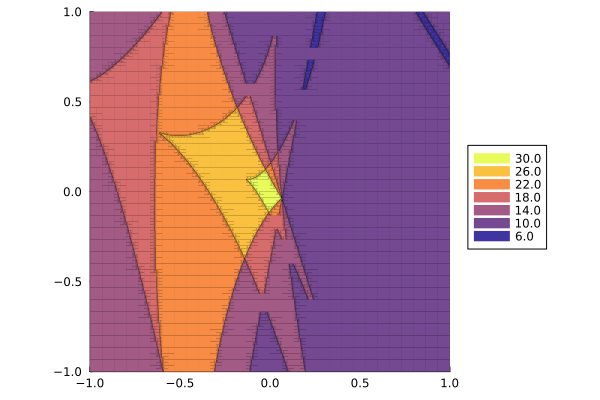}
    \end{minipage}
    \hfill
    \begin{minipage}{0.32\linewidth}
        \centering
        \includegraphics[width=\linewidth]{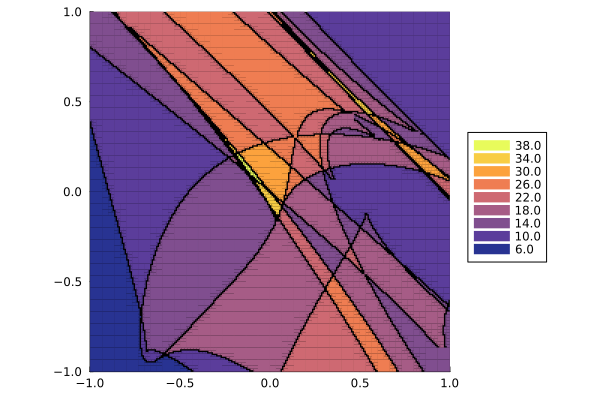}
    \end{minipage}
    \hfill
    \begin{minipage}{0.32\linewidth}
        \centering
    \includegraphics[width=\linewidth]{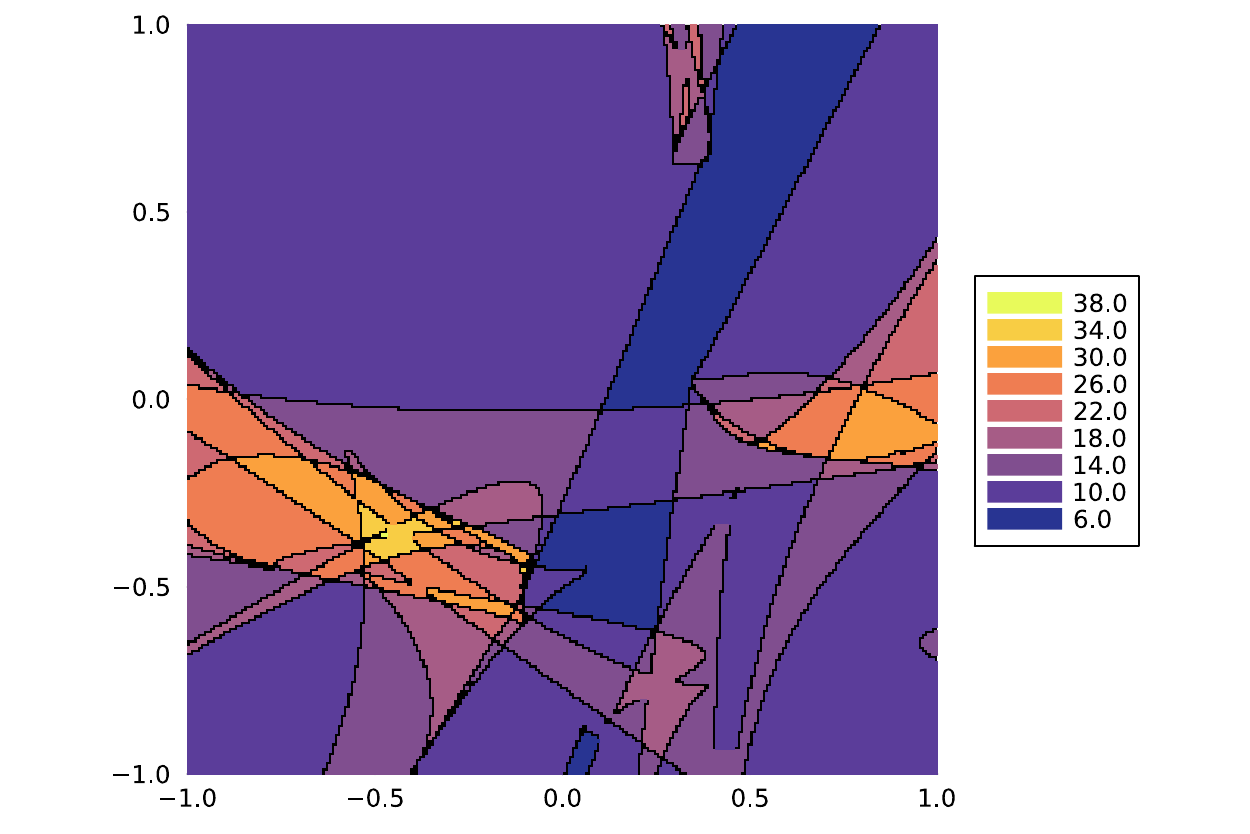}
    \end{minipage}
    \caption{ Parametric real solution maps for a MA($2$) process, from left to right: $n=3,4,5$.}
    \label{fig:MA2_real_a}
\end{figure}
\begin{figure}[H]
    \centering
    \begin{minipage}{0.32\linewidth}
        \centering
        \includegraphics[width=\linewidth]{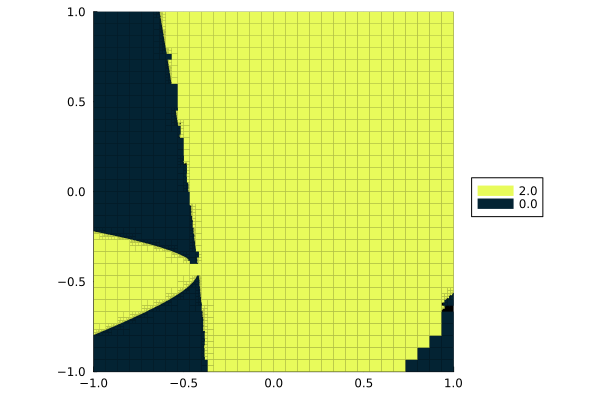}
    \end{minipage}
    \hfill
    \begin{minipage}{0.32\linewidth}
        \centering
        \includegraphics[width=\linewidth]{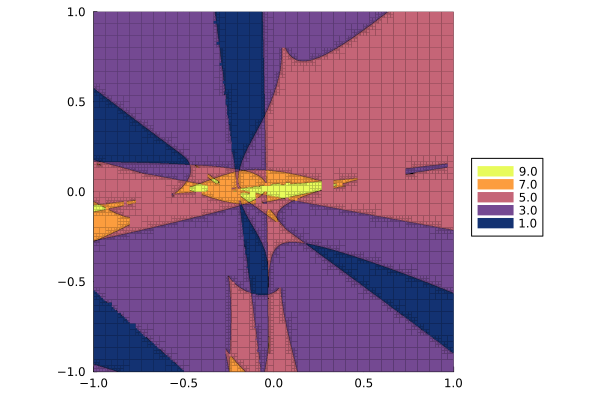}
    \end{minipage}
    \hfill
    \begin{minipage}{0.32\linewidth}
        \centering
        \includegraphics[width=\linewidth]{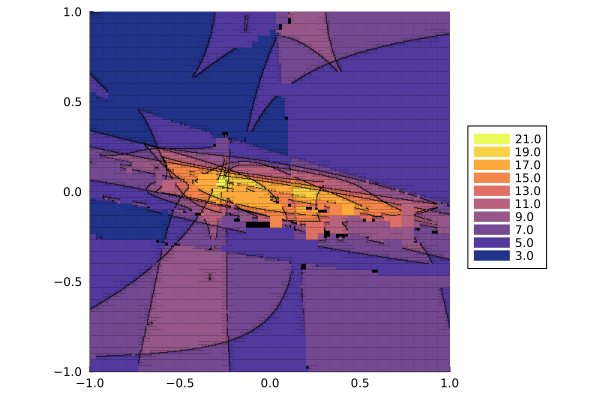}
    \end{minipage}
    \caption{Implicit real solution maps for a MA($2$) process, from left to right: $n=3,4,5$.}
    \label{fig:MA2_real_gamma}
\end{figure}

\begin{theorem}\label{thm:MA2n3}
 Consider a \emph{MA($2$)} process with $n = 3$ sample points $\bmf{y} = (y_1, y_2, y_3)^{\top}$. The implicit maximum likelihood degree is  $2$. The two (not necessarily real) solutions are given precisely by:
\begin{small}
\begin{align*}
     \gamma_0&=\dfrac{1}{6(y_1-2y_2+y_3)(y_1+2y_2+y_3)}\left(y_1^4-10y_2^4+2y_1^3y_3+2y_1y_2^2y_3+2y_1y_2y_3^3+2y_1y_3^3+y_3^4\right. \\ 
     &\left. \pm ((y_1+y_3)^2+2y_2^2)\sqrt{y_1^4-4y_1^2y_2^2+y_2^4+6y_1y_2^2y_3-y_1^2y_3^2-4y_2^2y_3^2+y_3^4} \right), \\ 
     \gamma_1&=y_2(y_1+y_3)\dfrac{(y_1^2-3y_2^2+y_1y_3+y_3^2)\mp \sqrt{y_1^4-4y_1^2y_2^2+y_2^4+6y_1y_2^2y_3-y_1^2y_3^2-4y_2^2y_3^2+y_3^4}}{2(y_1-2y_2+y_3)(y_1+2y_2+y_3)}, \\
     \gamma_2&=\dfrac{1}{2(y_1-2y_2+y_3)(y_1+2y_2+y_3)}\left(y_1^2y_2^2-2y_2^4+y_1^3y_3-4y_1y_2^2y_3+2y_1^2y_3^2+y_2^2y_3^2+y_1y_3^3\right. \\ 
     &\left. \mp 2y_2\sqrt{y_1^4-4y_1^2y_2^2+y_2^4+6y_1y_2^2y_3-y_1^2y_3^2-4y_2^2y_3^2+y_3^4}\right),
\end{align*}
\end{small}
with $(y_1-2y_2+y_3)(y_1+2y_2+y_3)\neq 0$. We also have the following special cases:
\begin{itemize}

\item \textbf{Special case I:} If $y_1  +2y_2 + y_3=0$  then the unique solution is
\begin{equation*}
\gamma_0 = \frac{y_3^2 + 2y_2y_3 - 5y_2^2}{6}, \quad
\gamma_1 = y_2^2, \quad
\gamma_2 = -\frac{y_3^2 + 2y_2y_3 + 3y_2^2}{2}.
\end{equation*}

\item \textbf{Special case II:} If $y_1 - 2y_2 + y_3=0$ then the unique solution is 
\begin{equation*}
\gamma_0 = \frac{y_3^2 - 2y_2y_3 - 5y_2^2}{6}, \quad
\gamma_1 = -y_2^2, \quad
\gamma_2 = -\frac{y_3^2 - 2y_2y_3 + 3y_2^2}{2}.
\end{equation*}

\end{itemize}

\end{theorem}

\begin{proof}
    A Gröbner basis computation reveals the following equations:
    \begin{small}
    \begin{align*}
12  (y_1-2y_2+y_3)(y_1+2y_2+y_3) \gamma_0^2+(-4y_1^4+40y_2^4-8y_1^3y_3-8y_1y_2^2y_3-8y_1^2y_3^2-8y_1y_3^3-4y_3^4)\gamma_0-\\3y_1^2y_2^4-8y_2^6+2y_1^3y_2^2y_3+2y_1y_2^4y_3+y_1^4y_3^2+4y_1^2y_2^2y_3^2-3y_2^4y_3^2+2y_1^3y_3^3+2y_1y_2^2y_3^3+y_1^2y_3^4=0\\
         (2y_1^2+4y_2^2+4y_1y_3+2y_3^2)\gamma_1+(-6y_1y_2-6y_2y_3)\gamma_0+y_1y_2^3-y_1^2y_2y_3+y_2^3y_3-y_1y_2y_3^2=0\\
y_2\gamma_2+(y_1+y_3)\gamma_1-3y_2\gamma_0+y_2^3-y_1y_2y_3=0
    \end{align*}
    \end{small}
    The first equation is quadratic in $\gamma_0$, leading to the two solutions via the quadratic formula. Once $\gamma_0$ is known, one can solve uniquely for $\gamma_1$ and $\gamma_2$ from the second and third equations, respectively, since these are linear equations.  

    For $y_2 = 0$, substituting into the original system, we obtain the following Gröbner basis:
    \begin{align*}
    12\gamma_0^2+4(-y_1^2-y_3^2)\gamma_0+y_1^2y_3^2=0 \\
        \gamma_1=0\\
        2\gamma_2-y_1y_3=0
    \end{align*}
which readily gives an expression consistent with the original specialized when $y_2=0$.
\end{proof}

\begin{example}
For the case $y_2 = 0$, we obtain the following solutions:
\begin{equation*}
    \gamma_0=\frac{1}{6}\left(y_1^2+y_3^2\pm \sqrt{y_1^4-y_1^2y_3^2+y_3^4}\right), \quad \gamma_1=0, \quad \gamma_2= \frac{y_1y_3}{2},
\end{equation*}
so we have that $\gamma_1 = 0$, which makes the model non-identifiable. This occurs under two distinct conditions:

\begin{itemize}
    \item \textbf{If $a_1 = 0$:} the \emph{MA($2$)} splits into two independent \emph{MA($1$)} processes --- one on even time indices and one on odd -- both sharing the same parameters $a_0,a_2$.
    
    \item \textbf{If $a_0 + a_2 = 0$:} a linear constraint among the coefficients again forces $\gamma_1 = 0$, yielding an autocovariance structure indistinguishable from the previous case.
\end{itemize}

\end{example}

\begin{remark}\label{rmk:bound2}
    As opposed to the \emph{MA($1$)} case (see Remark~\ref{rmk:bound1}), for \emph{MA($2$)}, the system does not attain the upper bound for the implicit ML degree computed in \cite[Theorem 2]{manivel2025proof}, which states that a generic model has an ML degree of $3n^2 - 9n + 7$. Already for $n=3$ we have $2< 3\cdot 3^2 - 9\cdot 3 + 7 = 7$.   
\end{remark}

\begin{remark}
For the likelihood function to be well-defined, it is necessary that the determinant of the matrix $\Sigma(\gamma)$ be nonzero. Hence, we require
\begin{equation*}
\det(\Sigma(\gamma)) = (\gamma_0-\gamma_2)\bigl(\gamma_0(\gamma_0+\gamma_2)-2\gamma_1^2\bigr)\neq 0.
\end{equation*}
Therefore, if for some combination of $y_1,y_2,y_3$ we obtain a solution satisfying $\det(\Sigma(\gamma))=0$, then such a solution cannot correspond to the maximum likelihood estimator.
\end{remark}

\begin{figure}[H]
    \centering
    \begin{minipage}{0.32\linewidth}
        \centering
        \includegraphics[width=\linewidth]{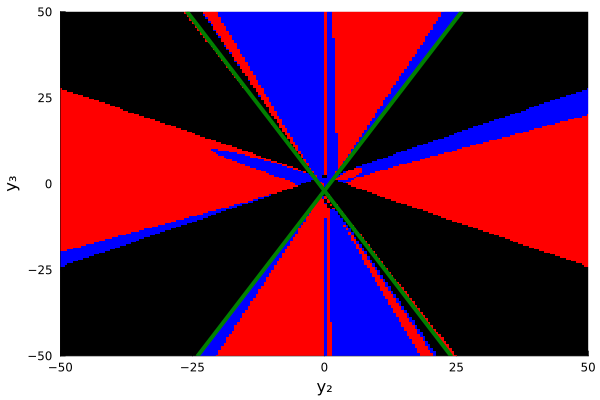}
    \end{minipage}
    \hfill
    \begin{minipage}{0.32\linewidth}
        \centering
        \includegraphics[width=\linewidth]{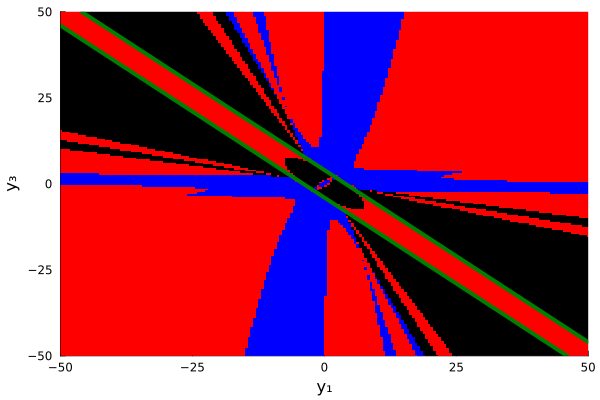}
    \end{minipage}
    \hfill
    \begin{minipage}{0.32\linewidth}
        \centering
    \includegraphics[width=\linewidth]{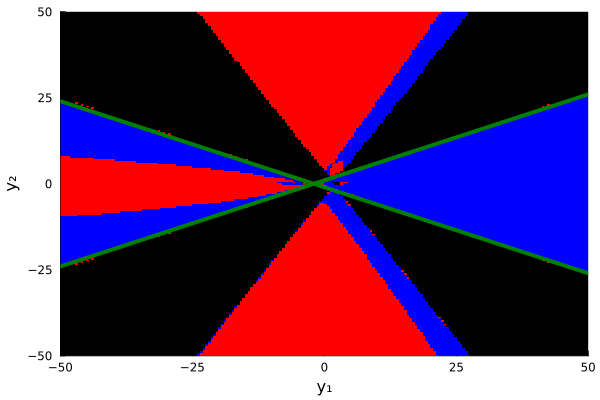}
    \end{minipage}
    \caption{ Solution maps for $n=3$ with one $y_i$ fixed to be 1. Blue indicates the first solution, red the second one, green a unique solution and black indicates that there are no real solutions. From left to right: fixed $y_1=2$, fixed $y_2=2$, fixed $y_3=2$.}
    \label{fig:MA2_n3_solution}
\end{figure}

Figure~\ref{fig:MA2_n3_solution} illustrates the maximum likelihood solution map for $n=3$ by fixing each coordinate $y_i=1$ in turn and mapping the existence and which of the two critical points in Theorem \ref{thm:MA2n3} is the maximum likelihood estimate across the remaining parameter space. Each panel partitions the plane into three distinct regions: first solution (blue), second solution (red), unique solution (green) and no real solution (black). The non-trivial boundaries between these regions reflect the nonlinear coupling inherent in the likelihood equations, providing a qualitative global picture of the solvability conditions that would be difficult to capture analytically.

\subsection{Non-invertible Solutions}

We have seen in the previous section that working with the implicit ML degree leads to a system of polynomial equations with significantly reduced algebraic complexity compared to the parametric ML degree. This reduction in complexity makes the resulting system substantially more tractable and allows computations for larger sample sizes.  However, this simplification comes at an important cost: the implicit ML Degree fails to identify non-invertible solutions. This is a relevant limitation, since in such cases the usual asymptotic properties of the estimators may break down, and standard asymptotic theory is no longer applicable \cite{sargan1983maximum,gospodinov2015minimum}. 

The underlying reason for this loss of solutions is that the MA($2$) model admits multiple preimages under the mapping from $(a_0,a_1,a_2)$ to $\gamma = (\gamma_0, \gamma_1, \gamma_2)$; in particular, this transformation is generically $8$-to-$1$. However, this is not true when the process is \emph{non-invertible}. These configurations are characterized by the vanishing of the Jacobian determinant of the transformation
\begin{equation*}
g(a_0,a_1,a_2) = (\gamma_0, \gamma_1, \gamma_2),
\end{equation*}
that is, $\det( J_g) = 0$. These degenerate configurations are excluded in the implicit ML degree computation, and therefore correspond to critical points that are present in the original $(a_0,a_1,a_2)$ parametrization but are not detected in the $\gamma$-space formulation.

To make this precise, consider the log-likelihood function $\ell(g(a_0,a_1,a_2))$. By the chain rule, the score equations in the new parametrization take the form
\begin{equation*}
[J_g(a_0,a_1,a_2)]^{\top} \nabla \ell(g(a_0,a_1,a_2)) = 0,
\end{equation*}
where $J_g$ is the Jacobian matrix of the transformation $g$. Explicitly, we have
\begin{equation}
[J_g]^{\top} =
\begin{pmatrix}
2a_0 & a_1 & a_2 \\
2a_1 & a_0 + a_2 & 0 \\
2a_2 & a_1 & a_0
\end{pmatrix}.
\end{equation}

To identify where this transformation becomes singular, we compute the determinant of this matrix. This is given by:
\begin{align*}
    \det([J_g]^{\top}) &= 2(a_0^3 - a_0 a_1^2 - a_2^3 - a_0 a_2^2 + a_0^2 a_2 + a_1^2 a_2) \\
    &= 2(a_0 - a_2)(a_0 - a_1 + a_2)(a_0 + a_1 + a_2).
\end{align*}

The vanishing of this determinant defines the locus where the Jacobian becomes singular, which, as we analyze next, corresponds exactly to the non-invertible regions of the MA($2$) model.

Since we search for solutions with $ \det(J_g^{\top})=0$ we have 3 options:
\begin{itemize}[leftmargin=*, itemindent=0pt]
 \item We first analyze the case $a_0 + a_2 = 0$ and $a_1 = 0$.  
Now the Jacobian becomes
\begin{equation*}
J_g^{\top} =
\begin{pmatrix}
2a_0 & 0 & -a_0\\[1pt]
0 & 0 & 0\\[1pt]
-2a_0 & 0 & a_0
\end{pmatrix}
\sim
\begin{pmatrix}
2a_0 & 0 & -a_0\\[1pt]
0 & 0 & 0\\[1pt]
0 & 0 & 0
\end{pmatrix},
\quad
\ker(J_g^{\top}) = 
\operatorname{span}\left\{
\begin{pmatrix}1\\[1pt]0\\[1pt]2\end{pmatrix},
\begin{pmatrix}0\\[1pt]1\\[1pt]0\end{pmatrix}
\right\}.
\end{equation*}

The induced autocovariances satisfy
\begin{equation*}
\gamma_0 = 2a_0^2, \qquad \gamma_1 = 0, \qquad \gamma_2 = -a_0^2,
\end{equation*}
yielding the simple constraint
\begin{equation*}
\gamma_0 + 2\gamma_2 = 0, \qquad 
\nabla g(\gamma_0, \gamma_1, \gamma_2) = 
\left(1,0,2\right)^{\top}.
\end{equation*}
The corresponding characteristic polynomial is
\begin{equation*}
\Theta(x) = \pm a_2 (x - 1)(x + 1),
\end{equation*}
which define  non-invertible MA($2$) process.

\item Asume $a_0 = a_2$ with $a_1 \neq 0$. In this situation, the Jacobian and its kernel simplify to
\begin{equation*}
J_g^{\top} \sim
\begin{pmatrix}
2a_0 & a_1 & a_0 \\
2a_1 & 2a_0 & 0 \\
0 & 0 & 0
\end{pmatrix}, \quad
\ker(J_g) = \operatorname{span} \left\{
\begin{pmatrix}a_0^2\\ -a_1a_0\\a_1^2-2a_0^2\end{pmatrix}
\right\}.
\end{equation*}
For this configuration, the autocovariances can be expressed as
\begin{equation*}
    \gamma_0 = 2a_0^2 + a_1^2, \quad
    \gamma_1 = 2 a_0 a_1, \quad
    \gamma_2 = a_0^2. 
\end{equation*}
Eliminating $a_0$ and $a_1$ leads to the constraint
\begin{equation*}
    4\gamma_0 \gamma_2 - 8\gamma_2^2 - \gamma_1^2 = 0, 
\end{equation*}
which characterizes the locus of non-invertible points in the $(\gamma_0, \gamma_1, \gamma_2)$ space.
To interpret this geometrically, recall that an MA($2$) process can be represented by its characteristic polynomial
$\Theta(x) = a_0 + a_1 x + a_2 x^2$.
When $a_0 = a_2$, this becomes
\begin{equation*}
    \Theta(x) = a_2 \left(x^2 + \frac{a_1}{a_2}x + 1\right)
    = a_2 (x - r_1)(x - \sfrac{1}{r_1}),
\end{equation*}
where $r_1$ and $1/r_1$ are reciprocal roots. This symmetry implies that one of the roots lies inside and the other outside the unit circle, yielding a non-invertible parametrization.
The alternative root configurations:
\begin{equation*}
\pm a_2 (r_1 x - 1)(x - \sfrac{1}{r_1})
\quad \text{and} \quad
\pm a_2 (x - r_1)(\sfrac{1}{r_1}x - 1)
\end{equation*}
break the condition $a_0 = a_2$, and therefore do not correspond to valid solutions.
Hence, the only two valid cases satisfying both the Jacobian degeneracy and the structural constraint are
\begin{equation*}
\pm a_2 (x - r_1)(x - \sfrac{1}{r_1}).
\end{equation*}
Note that if $|r_1|\neq 1$ the process admits the two non-invertible parametrizations: $\pm a_2 (x - r_1)(\sfrac{1}{r_1} x - 1)$ but these are not valid solutions of our system. So even if the MA($2$) process is invertible,  we only obtain non-invertible solutions.

\item $a_0 - a_1 + a_2 = 0$ with $a_1 \neq 0$, hence $a_0 = -a_1 - a_2$. Then
\begin{equation*}
J_g^{\top} =
\begin{pmatrix}2(a_1-a_2) & a_1 & a_2\\ 2a_1 & a_1 & 0\\ 2a_2 & a_1 & a_1-a_2\end{pmatrix}
\sim
\begin{pmatrix}2a_2 & 0 & a_2\\ 0 & a_1 & a_1\\ 0 & 0 & 0\end{pmatrix},
\ker(J_g^{\top}) = \operatorname{span}\left\{
\begin{pmatrix}1\\[1pt]2\\[1pt]-2\end{pmatrix}
\right\}.
\end{equation*}
with autocovariances
\begin{equation*}
\gamma_0 = 2(a_1^2 - a_1 a_2 + a_2^2),\quad
\gamma_1 = a_1^2,\quad
\gamma_2 = a_2(a_1- a_2).
\end{equation*}
Eliminating parameters yields the linear relation 
\begin{equation}
\gamma_0 - 2\gamma_1 + 2\gamma_2 = 0, \label{Eq:caso2_gamma}
\end{equation}
whose gradient is
\begin{equation}
\nabla g(\gamma_0, \gamma_1, \gamma_2) = (1,2,-2)^{\top}. \label{Eq:caso2_g}
\end{equation}
Writing the characteristic polynomial as $\Theta(x) = a_0 + a_1 x + a_2 x^2$, with $a_0 = a_1 - a_2$, its two factorizations are
\begin{equation*}
\pm a_2 (x - r_1)(x - r_2), \qquad \pm a_2 (r_1 x - 1)(r_2 x - 1).
\end{equation*}
Note that the mixed combinations $\pm a_2 (x - r_1)(r_2 x - 1)$ and $\pm a_2 (x - r_2)(r_1 x - 1)$ give coefficients that do not satisfy the constraint $a_0 = a_1 - a_2$, so they are not valid preimages for this case. Thus, the admissible root configurations for this locus are the ones coming from $\theta(x) = \pm a_2 (x - r_1)(x - r_2)$, which is consistent with the above $\gamma$–constraint.

Note again that if $|r_1|<1$ and $|r_2|>1$ (or vice versa) we will have that despite the model being invertible, the solutions obtained will only be non-invertible parameterizations.

\item $a_0 + a_1 + a_2 = 0$ (with $a_1 \neq 0$). The structure of this case is analogous to the previous one. Setting $a_0 = -a_1 - a_2$, the Jacobian reduces to
\begin{equation*}
J_g^{\top} =
\begin{pmatrix}
-2(a_1-a_2) & a_1 & a_2\\[1pt]
2a_1 & -a_1 & 0\\[1pt]
2a_2 & a_1 & -a_1-a_2
\end{pmatrix}
\sim
\begin{pmatrix}
-2a_1 & a_1 & 0\\[1pt]
0 & -a_1 & a_1\\[1pt]
0 & 0 & 0
\end{pmatrix}, \quad
\ker(J_g^{\top}) = \operatorname{span} \left\{
\begin{pmatrix}1\\[1pt]2\\[1pt]-2\end{pmatrix}
\right\}.
\end{equation*}
and the rest follows by the same argument: the autocovariances satisfy the linear relation
\begin{equation*}
\gamma_0 + 2\gamma_1 + 2\gamma_2 = 0, \qquad \nabla g(\gamma_0,\gamma_1,\gamma_2) = (1,2,2)^{\top},
\end{equation*}
mixed factorizations violate the constraint $a_0 = a_1 - a_2$ and are discarded, leaving $\pm a_2(x-r_1)(x-r_2)$ as the only valid root configuration. As before, if $|r_1|<1$ and $|r_2|>1$ (or vice versa), the solutions correspond only to non-invertible parametrizations.

\end{itemize}

Based on the preceding analysis, the solutions of the MA($2$) process can be rigorously categorized into five disjoint groups. One group corresponds to solutions that form part of the implicit ML Degree and is explicitly computable via the $\gamma$ parametrization from Theorem~\ref{thm:MA2n3}. The remaining four groups consist of solutions lying outside the implicit ML degree count, arising from singularities of $J_g^{\top}$.
Specifically, the solutions not captured by the implicit ML degree occur under the following conditions:
\begin{equation*}
a_0 = a_2, \quad 
a_0 - a_1 + a_2 = 0, \quad 
a_0 + a_1 + a_2 = 0, \quad 
a_1 = 0,
\end{equation*}
where in the first three cases we assume $a_1 \neq 0$. These conditions correspond precisely to the points where the Jacobian becomes singular, yielding additional solutions that lie beyond the scope of the implicit ML degree and are not captured when the likelihood is optimized solely with respect to the $\gamma$ parametrization.

We computed all solutions numerically with respect to the parameters $a_i$, classifying them according to the five cases described above. The total counts obtained for each group are summarized in Table~\ref{tab:MA(2)gammares}.

\begin{table}[H]
    \centering
    \begin{tabular}{|c|c|c|c|c|c|c|}
        \hline
        $n$ & $g^{-1}(\gamma)$ & $a_1=0$ & $a_0=a_2$ & $a_0-a_1+a_2=0$ & $a_0+a_1+a_2=0$ & ML \\
         &  & $a_0+a_2=0$ & $a_1\neq 0$ & $a_1\neq 0$ & $a_1\neq 0$ & degree \\
        \hline
        3 & 16 & 2 & 16 & 12 & 12 & 58 \\
        4 & 72 & 2 & 24 & 20 & 20 & 138 \\
        5 & 168 & 2 & 32 & 28 & 28 & 258 \\
        6 & 296 & 2 & 40 & 36 & 36 & 410 \\
        7 & 440 & 2 & 48 & 44 & 44 & 578 \\
        8 & 616 & 2 & 56 & 52 & 52 & 778 \\
        9 & 768 & 2 & 64 & 60 & 60 & 954 \\
        10 & 928 & 2 & 72 & 68 & 68 & 1136 \\
        \hline
    \end{tabular}
    \caption{
    Parametric ML degree for the MA($2$) process with respect to the sample size $n$. Decomposed into the implicit ML degree 
    (column $g^{-1}(\gamma)$) and four additional solution groups arising from singularities of $J_g^{\top}$.
    }
    \label{tab:MA(2)gammares}
\end{table}
From these results, we observe a clear pattern in the solutions with the singularities of $J_g^{\top}$ , which motivates the following theorem:
\begin{theorem}\label{teo:soluciones_noinv_ma2}
The solutions of the parametric ML degree with $\det(J_g^{\top})=0$ are exactly $24n-30$, of which 
\begin{enumerate}
    \item[i)]  $2$ solutions with $a_1 = 0$ and $a_0 =- a_2$.
    \item[ii)] $2\cdot (4(n-1))$ solutions with $a_0 = a_2$ and $a_1 \neq 0$.
    \item[iii)] $4(2n-3)$ solutions satisfying $a_0 - a_1 + a_2 = 0$ with $a_1 \neq 0$.
    \item[iv)] $4(2n-3)$ solutions satisfying $a_0 + a_1 + a_2 = 0$ with $a_1 \neq 0$.

\end{enumerate}
Moreover for item~\emph{i)},  given a sample $y_1, y_2, \dots, y_n$ with $\bmf{y}_{\mathrm{odd}} = (\bmf{y}_{1},\bmf{y}_{3},\dots, \bmf{y}_{2k+1},\dots)^{\top}$ and $\bmf{y}_{\mathrm{even}} = (\bmf{y}_{2},\bmf{y}_{4},\dots, \bmf{y}_{2k},\dots)^{\top}$ denoting the sub-vectors of odd- and even-indexed observations, respectively, the explicit solutions are determined by

{\footnotesize
\[
a_0
= \pm \frac{1}{\sqrt{n}}\sqrt{
\sum_{i,j=1}^{k_1} \frac{\min(i,j)\,[k_1+1-\max(i,j)]}{k_1+1}\, (\bmf{y}_{\mathrm{odd}})_i\, (\bmf{y}_{\mathrm{odd}})_j
+
\sum_{i,j=1}^{k_2} \frac{\min(i,j)\,[k_2+1-\max(i,j)]}{k_2+1}\, (\bmf{y}_{\mathrm{even}})_i\, (\bmf{y}_{\mathrm{even}})_j
}
\]
}

where $k_1 = \lceil n/2 \rceil$, $k_2 = \lfloor n/2 \rfloor$.
\end{theorem}

\begin{proof}

\begin{enumerate}[leftmargin=*, itemindent=0pt]

\item[i)] In the case where $a_1 = 0$ and $a_0 = -a_2$, we have
$\dfrac{\partial \Sigma}{\partial a_1} = 0 \quad \text{and} \quad 
\dfrac{\partial \Sigma}{\partial a_0} = -\dfrac{\partial \Sigma}{\partial a_2}.$

Consequently, the derivatives of the log-likelihood yield a system with a single equation in a single variable.  The covariance matrix in (\ref{Eq:MA2_Sigma_a}) can be written as
\begin{equation*}
\Sigma = a_0^2 
\begin{bmatrix}
2 & 0 & -1 & 0 & \cdots & 0 \\
0 & 2 & 0 & -1 & \cdots & 0 \\
-1 & 0 & 2 & 0 & \cdots & 0 \\
0 & -1 & 0 & 2 & \cdots & 0 \\
\vdots & \vdots & \vdots & \vdots & \ddots & \vdots \\
0 & 0 & 0 & 0 & \cdots & 2
\end{bmatrix} := a_0^2 M,  
\end{equation*}
with derivative
$
\dfrac{\partial \Sigma}{\partial a_0}= 2a_0 M.
$
Considering the log-likelihood \eqref{eq:likMA}, setting $ \dfrac{\partial \ell}{ \partial a_0} = 0$ gives 
\begin{equation*}
\frac{\partial \ell}{\partial a_0} 
= \operatorname{tr}\left(a_0^{-1} M^{-1} M\right) - \frac{1}{a_0^3} \bmf{y}^{\top} M^{-1} M M^{-1} \bmf{y} = 0  \; \Rightarrow \;
a_0 = \pm \sqrt{\frac{1}{n}\bmf{y}^{\top} M^{-1} \bmf{y}}.
\end{equation*}

Since $M$ is diagonally dominant, $M$ and its inverse $M^{-1}$ are positive definite, so that $\bmf{y}^{\top} M^{-1} \bmf{y} \geq  0$, ensuring $a_0 \in \mathbb{R}$. Moreover, $a_0$ can be computed explicitly: define the permutation matrix $P$ corresponding to reordering the indices
$
(1,2,3,\dots,n)\mapsto (1,3,5,\dots,2,4,6,\dots).
$
Since $P$ is orthogonal, we have $P^{-1} = P^{\top}$. Let $ \widetilde M := P^{\top} M P$
denote the permuted matrix. Because $M_{ij}\neq 0$ only when $i$ and $j$ have the same parity, the matrix $\widetilde M$ is block diagonal:
\begin{equation*}
\widetilde M =
\begin{bmatrix}
T_{k_1} & 0\\
0 & T_{k_2}
\end{bmatrix},
\end{equation*}
where  $k_1 = \lceil n/2 \rceil$, $k_2 = \lfloor n/2 \rfloor$ and $T_k$ denotes the tridiagonal Toeplitz $k \times k$ matrix
\begin{small}
\begin{equation*}
T_k =
\begin{bmatrix}
2 & -1 & 0 & \cdots & 0\\
-1 & 2 & -1 & \cdots & 0\\
0 & -1 & 2 & \cdots & 0\\
\vdots & \vdots & \vdots & \ddots & \vdots\\
0 & 0 & 0 & \cdots & 2
\end{bmatrix}.
\end{equation*}
\end{small}
Since $P$ is orthogonal,
$
M^{-1} = P\, \widetilde M^{-1} P^{\top}.
$ Moreover, as $\widetilde M$ is block diagonal,
\begin{equation*}
\widetilde M^{-1} =
\begin{bmatrix}
T_{k_1}^{-1} & 0\\
0 & T_{k_2}^{-1}
\end{bmatrix}.
\end{equation*}
Let $\bmf{y} \in \mathbb{R}^n$ and define the permuted vector
$
\widetilde{\bmf{y}} := P^{\top} \bmf{y} = (
\bmf{y}_{\mathrm{odd}},\bmf{y}_{\mathrm{even}})^{\top}$.
Then,
\begin{equation*}
\bmf{y}^{\top} M^{-1} \bmf{y}
= \bmf{y}^{\top} P
\begin{bmatrix}
T_{k_1}^{-1} & 0 \\
0 & T_{k_2}^{-1}
\end{bmatrix}
P^{\top} \bmf{y} = \widetilde{\bmf{y}} ^{\top}
\begin{bmatrix}
T_{k_1}^{-1} & 0\\
0 & T_{k_2}^{-1}
\end{bmatrix}
\widetilde{\bmf{y}} = \bmf{y}_{\mathrm{odd}}^{\top} T_{k_1}^{-1} \bmf{y}_{\mathrm{odd}}
+ \bmf{y}_{\mathrm{even}}^{\top} T_{k_1}^{-1} \bmf{y}_{\mathrm{even}}.
\end{equation*}
By Lemma~\ref{Lemma:inverse_toeplitz}, the inverse of the tridiagonal Toeplitz matrix $T\in\mathbb{R}^{k\times k}$ is given by 
\begin{equation*}
(T^{-1})_{ij} = \frac{\min(i,j)\,[\,k+1-\max(i,j)\,]}{k+1}, \qquad i,j = 1,\dots,k.
\end{equation*}
Substituting this into the quadratic form yields 
{\footnotesize
\begin{equation*}
a_0^2
= 
\frac{1}{n}\left(\sum_{i,j=1}^{k_1} \frac{\min(i,j)\,[k_1+1-\max(i,j)]}{k_1+1}\, (\bmf{y}_{\mathrm{odd}})_i (\bmf{y}_{\mathrm{odd}})_{j}
+
\sum_{i,j=1}^{k_2} \frac{\min(i,j)\,[k_2+1-\max(i,j)]}{k_2+1}\, (\bmf{y}_{\mathrm{even}})_i (\bmf{y}_{\mathrm{even}})_j\right).
\end{equation*}
}

\item[ii)]  In the case $a_0 = a_2$, we analyze the system in terms of $a_0$ and $a_1$. 
Replacing in the covariance matrix in (\ref{Eq:MA2_Sigma_a}) we obtain: 
\begin{equation*}
\Sigma = 
\begin{bmatrix}
2a_0^2+a_1^2 & 2a_0a_1 & a_0^2 & 0 & \cdots & 0 \\
2a_0a_1 & 2a_0^2+a_1^2 & 2a_0a_1 & a_0^2 & \cdots & 0 \\
a_0^2 & 2a_0a_1 & 2a_0^2+a_1^2 & 2a_0a_1 & \cdots & 0 \\
0 & a_0^2 & 2a_0a_1 & 2a_0^2+a_1^2 & \cdots & 0 \\
\vdots & \vdots & \vdots & \vdots & \ddots & \vdots \\
0 & 0 & 0 & 0 & \cdots & 2a_0^2+a_1^2
\end{bmatrix} 
\end{equation*}
and replacing in \eqref{Eq:MA2_Sigma_derivate} and \eqref{Eq:MA2_Sigma_derivate2} we obtain the matrices of derivatives:
\begin{equation*}
   \dfrac{\partial \Sigma}{\partial a_0}  = 
\begin{bmatrix}
2a_0 & a_1 & a_0 & 0 & \cdots & 0 \\
a_1 & 2a_0 & a_1 &  a_0 & \cdots & 0 \\
a_0 & a_1 & 2a_0 & a_1 & \cdots & 0 \\
0 & a_0 & a_1 & 2a_0 & \cdots & 0 \\
\vdots & \vdots & \vdots & \vdots & \ddots & \vdots \\
0 & 0 & 0 & 0 & \cdots & 2a_0
\end{bmatrix} \qquad  \dfrac{\partial \Sigma}{\partial a_1}  = 2 
\begin{bmatrix}
a_1 & a_0 & 0 & 0 & \cdots & 0 \\
a_0 & a_1 & a_0 & 0 & \cdots & 0 \\
0 & a_0 & a_1 & a_0 & \cdots & 0 \\
0 & 0 & a_0 & a_1 & \cdots & 0 \\
\vdots & \vdots & \vdots & \vdots & \ddots & \vdots \\
0 & 0 & 0 & 0 & \cdots & a_1
\end{bmatrix}
\end{equation*}

We introduce the change of variable $a_1 = b a_0$, so that the three matrices above become
\begin{equation*}
\Sigma = a_0^2 
\begin{bmatrix}
2+b^2 & 2b & 1 & 0 & \cdots & 0 \\
2b & 2+b^2 & 2b & 1 & \cdots & 0 \\
1 & 2b & 2+b^2 & 2b & \cdots & 0 \\
0 & 1 & 2b & 2+b^2 & \cdots & 0 \\
\vdots & \vdots & \vdots & \vdots & \ddots & \vdots \\
0 & 0 & 0 & 0 & \cdots & 2+b^2
\end{bmatrix} =: a_0^2 D,
\end{equation*}
\begin{equation*}
   \dfrac{\partial \Sigma}{\partial a_0}  = 2a_0 
\begin{bmatrix}
2 & b & 1 & 0 & \cdots & 0 \\
b & 2 & b & 1 & \cdots & 0 \\
1 & b & 2 & b & \cdots & 0 \\
0 & 1 & 2 & 2 & \cdots & 1 \\
\vdots & \vdots & \vdots & \vdots & \ddots & \vdots \\
0 & 0 & 0 & 0 & \cdots & 2
\end{bmatrix}=: 2 a_0 D_0, \quad  \dfrac{\partial \Sigma}{\partial a_1}  = 2a_0 
\begin{bmatrix}
b & 1 & 0 & 0 & \cdots & 0 \\
1 & b & 1 & 0 & \cdots & 0 \\
0 & 1 & b & 1 & \cdots & 0 \\
0 & 0 & 1 & b & \cdots & 0 \\
\vdots & \vdots & \vdots & \vdots & \ddots & \vdots \\
0 & 0 & 0 & 0 & \cdots & b
\end{bmatrix}=:2 a_0 D_1.
\end{equation*}
Substituting into \eqref{eq:loglikMA}, the equation $\dfrac{\partial \ell}{\partial a_1} = 0$ for $\operatorname{tr}(D^{-1} D_1) \neq 0$ becomes 
\begin{equation*}
\operatorname{tr}\left(\frac{1}{a_0^2} D^{-1} 2 a_0 D_1 \right) - \bmf{y}^\top \frac{1}{a_0^2} D^{-1} 2 a_0 D_1 \frac{1}{a_0^2} D^{-1} \bmf{y} = 0 \; \Rightarrow \; a_0^2 = \frac{\bmf{y}^\top D^{-1} D_1 D^{-1} \bmf{y}}{\operatorname{tr}(D^{-1} D_1)}.
\end{equation*}
Substituting into $\dfrac{\partial \ell}{\partial a_0} = 0$ gives $a_0^2 \operatorname{tr}(D^{-1} 2D_0) - \bmf{y}^\top D^{-1} 2D_0 D^{-1} \bmf{y} = 0$. That is, 
\begin{equation}\label{eq:trdprime}
\operatorname{tr}(D' D_0) \bmf{y}^\top D' D_1 D' \bmf{y} - \operatorname{tr}(D' D_1) \bmf{y}^\top D' D_0 D' \bmf{y} = 0.
\end{equation}
with $D' = \operatorname{adj}(D) = \det(D) D^{-1}$. Since $D$ is a pentadiagonal Toeplitz matrix, expanding $\det(D_{(n)})$
along the last row and column via cofactor expansion yields the
fifth-order linear recurrence (a standard technique for banded Toeplitz
matrices; cf.~\cite{elouafi2013note})
\begin{align*}
\det(D_{(n)}) &= (1+b^2)\det(D_{(n-1)}) + (2-3b^2)\det(D_{(n-2)})\\
&- (2-3b^2)\det(D_{(n-3)}) - (1+b^2)\det(D_{(n-4)}) + \det(D_{(n-5)}).
\end{align*}

where $D_{(n)}$ is the matrix $D$ of size $n \times n$. By strong induction, $\det(D_{(n)})$ is of degree $2n$, so entries of $D'$ have degree at most $2(n-1)$. Thus, the total degree of \eqref{eq:trdprime} is $6(n-1)+1$, with no
cancellation at the highest degree, as can be verified directly by
taking $\mathbf{y}=\mathbf{1}$. Since we cleared the denominator $\operatorname{tr}(D'(D_1)) \neq 0$, which has degree $2(n-1) + 1$, the equation $\dfrac{\partial \ell}{\partial a_0} = 0$ yields $4(n-1)$ solutions for $b$. Including the two possible solutions for $a_0$, we conclude the system admits generically a total of $8(n-1)$ solutions. 

\item[iii)] In the case $a_0 - a_1 + a_2 = 0$, by \eqref{Eq:caso2_gamma} and \eqref{Eq:caso2_g} the system becomes
\begin{equation*}
\nabla \ell (\gamma_0, \gamma_1, \gamma_2) = \lambda 
\left(1,2,-2 \right)^{\top}, \quad 
\gamma_0 = 2\gamma_1 - 2\gamma_2.
\end{equation*}
In particular, $\dfrac{\partial \ell}{\partial \gamma_0} = \lambda$, that is, $\operatorname{tr}(\Sigma^{-1}) - \bmf{y}^{\top} \Sigma^{-1} \Sigma^{-1} \bmf{y} = \lambda$. Introducing the change of variable $\gamma_1 = b \gamma_2$, the covariance matrix in \eqref{Eq:MA2_Sigma_gamma} can be written as
\begin{equation*}
\Sigma = \gamma_2 
\begin{bmatrix}
2(b-1) & b & 1 & 0 & \cdots & 0 \\
b & 2(b-1) & b & 1 & \cdots & 0 \\
1 & b & 2(b-1) & b & \cdots & 0 \\
0 & 1 & b & 2(b-1) & \cdots & 0 \\
\vdots & \vdots & \vdots & \vdots & \ddots & \vdots\\
0 & 0 & 0 & 0 & \cdots & 2(b-1)
\end{bmatrix} := \gamma_2 M.
\end{equation*}

Denoting $M_1 = \dfrac{ \partial \Sigma}{\partial \gamma_1}$ in \eqref{Eq:MA2_Sigma_derivate} , the equation $\dfrac{\partial \ell}{\partial \gamma_1} = 2\lambda$  becomes
\begin{align*}
\frac{1}{\gamma_2} \operatorname{tr}(M^{-1} M_1) - \frac{1}{\gamma_2^2} \bmf{y}^{\top} M^{-1} M_1 M^{-1} \bmf{y} 
&= 2 \left( \frac{1}{\gamma_2} \operatorname{tr}(M^{-1}) - \frac{1}{\gamma_2^2} \bmf{y}^{\top} M^{-1} M^{-1} \bmf{y} \right),
\end{align*}
so that 
$$\gamma_2 = \frac{\bmf{y}^{\top} M^{-1} (M_1 - 2I) M^{-1} \bmf{y}}{\operatorname{tr}(M^{-1} (M_1 - 2I))}.$$
Similarly, denoting $M_2 = \dfrac{\partial \Sigma}{\partial \gamma_2}$ in \eqref{Eq:MA2_Sigma_derivate2}, the equation $\dfrac{\partial \ell}{\partial \gamma_1} = -2\lambda$ becomes 
\begin{align*}
0 &= \gamma_2 \operatorname{tr}(M^{-1}(M_2 - 2I)) - \bmf{y}^{\top} M^{-1} (M_2 + 2I) M^{-1} \bmf{y}  \\
&= \dfrac{\operatorname{tr}(M^{-1}(M_2 + 2I))}{\operatorname{tr}(M^{-1}(M_1 - 2I))} \bmf{y}^{\top} M^{-1} (M_1 - 2I) M^{-1} \bmf{y}
- \bmf{y}^{\top} M^{-1} (M_2 + 2I) M^{-1} \bmf{y}.
\end{align*} 

Defining $M' = \operatorname{adj}(M) = \det(M) M^{-1}$ and multiplying through by $\det(M)^2$, we obtain
\begin{align*}
\operatorname{tr}(M' (M_2 + 2I)) \bmf{y}^{\top} M' (M_1 - 2I) M' \bmf{y}
- \operatorname{tr}(M' (M_1 - 2I)) \bmf{y}^{\top} M' (M_2 + 2I) M' \bmf{y} &= 0.
\end{align*}
As in the proof of (ii), $M$ is a pentadiagonal Toeplitz matrix, so
expanding $\det(M_{(n)})$ along the last row and column via cofactor
expansion yields the fifth-order linear recurrence (a standard
technique for banded Toeplitz matrices; cf.~\cite{elouafi2013note})
\begin{align*}
\det(M_{(n)}) &= (2b-3)\det(M_{(n-1)}) - (b^2-2b+2)\det(M_{(n-2)}) \\
&+ (b^2-2b+2)\det(M_{(n-3)}) + (3-2b)\det(M_{(n-4)}) + \det(M_{(n-5)}),
\end{align*}
where $M_{(n)}$ is the $n \times n$ principal submatrix of $M$. By strong
induction, $\det(M_{(n)})$ is of degree $n$, so entries of $M'$ have
degree at most $n-1$. Thus the total degree is $3(n-1)$, with no
cancellation at the highest degree, as can be verified directly by
taking $\mathbf{y}=\mathbf{1}$.
Since $\operatorname{tr}(M'(M_1 - 2I)) \neq 0$ and has degree $n-1$, the
final expression has $2(n-1)$ solutions.

Including the trivial solution with $a_1 = 0$, the total number of solutions is $2n-3$. Furthermore, there exist four distinct solutions for $(a_0, a_1, a_2)$ that satisfy the given restriction, leading to a total of $4(2n-3)$ solutions.

The proof of iv) (i.e. the case $a_0 = -a_1 - a_2$) is analogous.\qedhere
\end{enumerate}\end{proof}

\subsection{Maximum at a Non-Invertible Point}\label{Section:maximo_no_invertible}
Recall from the introduction that the invertibility of a MA process can be characterized by the nature of the roots of the associated characteristic polynomial. It is possible that a local maximum of the likelihood function corresponds to a non-invertible process, even when the true data-generating process is invertible.  
\cite{cryer1981small} and \cite{anderson1986noninvertible} showed via simulation that the likelihood function can be bimodal: besides the maximum near the true invertible value, a second local maximum can appear at a non-invertible point and, in small samples, exceed the invertible one in likelihood value. This is more likely the closer the true parameter is to the invertibility boundary, where the surface flattens and both maxima become close in height. We now extend this analysis to the MA($2$) process.

Consider the MA($2$) process
\begin{equation*}
Y_t = a_0 Z_t + a_1 Z_{t-1} + a_2 Z_{t-2},
\end{equation*}
where $(Z_t)_{t \in \mathbb{Z}}$ are i.i.d. standard normal and $a_i \in \mathbb{R}$.

For this section, it is convenient to reparametrize the model in terms of $\theta_1 = a_1/a_0$, $\theta_2 = a_2/a_0$, and $\sigma^2 = a_0^2$. Under this change of variables, the process can be written as
\begin{equation*}
Y_t = \sigma\left(Z_t + \theta_1 Z_{t-1} + \theta_2 Z_{t-2}\right),
\qquad \theta_i \in \mathbb{R}, \ \sigma > 0.
\end{equation*}
Let $\bmf{y} = (y_1, \ldots, y_n)^{\top}$ be an observed sample vector. The autocovariance matrix of $Y_t$ can be expressed as
\begin{equation*}
\Sigma = \sigma^2_0 R, \quad \text{where } \,   \sigma^2_0 = \sigma^2 (1 + \theta_1^2 + \theta_2^2),
\end{equation*}
and $R$ is the autocorrelation matrix, which is Toeplitz satisfying $R_{ii} = 1$ for all $i$, and
\begin{equation*}
R_{i,i+1} = R_{i+1,i} = \frac{\theta_1(1 + \theta_2)}{1 + \theta_1^2 + \theta_2^2}, 
\qquad
R_{i,i+2} = R_{i+2,i} = \frac{\theta_2}{1 + \theta_1^2 + \theta_2^2},
\end{equation*}
with all remaining entries equal to zero.
This reparametrization allows the log-likelihood function \eqref{eq:likMA} to be written in terms of $\theta_1, \theta_2$ and $\sigma^2_0$ as
\begin{equation*}
\ell(\theta_1, \theta_2, \sigma^2_0) = -\frac{N}{2} \log \sigma^2_0 - \frac{1}{2} \log \det(R) - \frac{1}{2\sigma^2_0} \bmf{y}^\top R^{-1} \bmf{y},
\end{equation*}
Maximizing with respect to $\sigma^2_0$ gives
\begin{equation*}
\widehat{\sigma}^2_0 = \frac{1}{n} \bmf{y}^{\top} R^{-1} \bmf{y},
\end{equation*}
which, substituted back, yields the \emph{concentrated likelihood function} \cite{anderson1986noninvertible}:
\begin{equation*}
M(\theta_1, \theta_2) = -\log|R| + n \log\!\left(\bmf{y}^{\top} R^{-1} \bmf{y}\right).
\end{equation*}
To investigate the possibility of a local maximum in the non-invertible region, we analyze the Hessian matrix of \(M\).  
In compact form, the elements of the Hessian are given by
\begin{align*}
H_{ij} =& -\mathrm{tr}\!\Big(
 -R^{-1}\frac{\partial R}{\partial \theta_i} R^{-1}\frac{\partial R}{\partial \theta_j}
   + R^{-1}\frac{\partial^2 R}{\partial \theta_i \partial \theta_j}
\Big)  + n \frac{1}{\bmf{y}^{\top} R^{-1} \bmf{y}} \Big[
-2\,\bmf{y}^{\top} R^{-1} \frac{\partial R}{\partial \theta_i} R^{-1} \frac{\partial R}{\partial \theta_j} R^{-1} \bmf{y}
\\ &+ \bmf{y}^{\top} R^{-1} \frac{\partial^2 R}{\partial \theta_i \partial \theta_j} R^{-1} \bmf{y} 
+ (\bmf{y}^{\top} R^{-1} \frac{\partial R}{\partial \theta_i} R^{-1} \bmf{y})(\bmf{y}^{\top} R^{-1} \frac{\partial R}{\partial \theta_j} R^{-1} \bmf{y})
\Big],
\end{align*}
and the condition for a local maximum is that $H$ be negative definite.

Consider the point \((\theta_1, \theta_2) = (0, -1)\), which corresponds to the the non-invertible case i) where $a_0+a_2=0$ and $a_1=0$. Evaluating at the Hessian we obtain 
\begin{align*}
H_{11} &= -\mathrm{tr}\!\left(R^{-1}\frac{\partial^2 R}{\partial \theta_1^2}\right)
+ n \frac{\bmf{y}^{\top} R^{-1} \frac{\partial^2 R}{\partial \theta_1^2} R^{-1} \bmf{y}}{\bmf{y}^{\top} R^{-1} \bmf{y}},\\
H_{12} &= -\mathrm{tr}\!\left(R^{-1}\frac{\partial^2 R}{\partial \theta_1 \partial \theta_2}\right)
+ n \frac{\bmf{y}^{\top} R^{-1} \frac{\partial^2 R}{\partial \theta_1 \partial \theta_2} R^{-1} y}{\bmf{y}^{\top} R^{-1} \bmf{y}}.
\end{align*}

The computation of \(H_{22}\) requires the spectral decomposition of block-tridiagonal matrices, which can be separated into two submatrices of dimensions $\lfloor n/2 \rfloor$ and $\lceil n/2 \rceil$. The eigenvalues of these blocks are
\begin{equation*}
\lambda_k = 1-\cos\!\Big(\frac{k \pi}{\lfloor n/2\rfloor+1}\Big),\qquad
\lambda'_k = 1-\cos\!\Big(\frac{k \pi}{\lceil n/2\rceil+1}\Big),
\end{equation*}
with corresponding sine-based eigenvectors.  
Letting $\Lambda$ denote the diagonal matrix of eigenvalues of $R$ and $\Lambda'$ the diagonal matrix of eigenvalues of $\frac{\partial^2 R}{\partial \theta_2^2}(0,-1)$, we can write
\begin{equation*}
H_{22} = -\mathrm{tr}\!\big(-\Lambda^{-2} (\Lambda')^2 + \Lambda^{-1} \Lambda'\big)
+ n \frac{1}{\mathbf{z}^\top \Lambda^{-1} \mathbf{z}} \Big[-2\, \mathbf{z}^\top \Lambda^{-3} (\Lambda')^2 \mathbf{z} + \mathbf{z}^\top \Lambda^{-2} \Lambda' \mathbf{z} + (\mathbf{z}^\top \Lambda^{-2} \Lambda' \mathbf{z})^2\Big],
\end{equation*}
where $z$ is the representation of $\mathbf{y}$ in the eigenbasis of $R$.

Finally, the probability of obtaining a local maximum at the non-invertible point $(\theta_1, \theta_2) = (0, -1)$ was evaluated via Monte Carlo simulation for several sample sizes and values of $\theta_2^*$.  
The results show that this probability is substantial in small samples and decreases as $n$ increases, except for the degenerate case $\theta_2 = -1$.  
These results are summarized in Table~\ref{tab:diagbox}.

\begin{table}[H]
\centering
\begin{tabular}{|l|c|c|c|c|c|}
\toprule
\diagbox[width=5em]{$\theta_2^*$}{$n$} & \textbf{10} & \textbf{25} & \textbf{50} & \textbf{75} & \textbf{100} \\
\midrule
 $0$ & $9.84\%$ & $1.03\%$ & $0.05\%$ & $0.00\%$ & $0.00\%$ \\
 $-0.2$ & $15.37\%$ & $2.35\%$ & $0.15\%$ & $0.01\%$ & $0.00\%$ \\
 $-0.4$ & $25.16\%$ & $5.13\%$ & $0.62\%$ & $0.05\%$ & $0.01\%$ \\
 $-0.6$ & $40.29\%$ & $15.17\%$ & $2.75\%$ & $0.45\%$ & $0.02\%$ \\
 $-0.8$ & $53.31\%$ & $38.81\%$ & $16.91\%$ & $7.70\%$ & $3.37\%$ \\
 $-1$   & $58.20\%$ & $60.20\%$ & $61.34\%$ & $61.04\%$ & $61.19\%$ \\
\bottomrule
\end{tabular}
\caption{Probability of observing a local maximum at $(\theta_1,\theta_2)=(0,-1)$ with data simulated from an MA(2) with parameter $(0,\theta_2^*)$.}
\label{tab:diagbox}
\end{table}

\noindent
In summary, consistent with the findings of \cite{cryer1981small,anderson1986noninvertible} for the MA($1$) process, our results show that the likelihood function can attain maxima parameter values corresponding to non-invertible processes even when the underlying process is invertible.  
This property should be considered when estimating MA($2$) models in small samples and when interpreting maximum likelihood estimates near the boundary of invertibility.

\subsection{Simulation Study of Parameter Estimation}\label{section:Estimation_MA_process}

We conduct a simulation study to compare our numerical algebraic geometry approach  against a standard numerical maximization approach. 
A total of 500 independent samples are generated, each consisting of $n = 10$ observations drawn from an MA($2$) process defined by
\begin{equation*}
Y_t = a_0 Z_t + a_1 Z_{t-1} + a_2 Z_{t-2},
\qquad Z_t \sim N(0,1),
\end{equation*}
with parameters fixed at $a_0 = 1$, $a_1 = 0.5$, and $a_2 = -0.3$.

The model parameters are estimated via maximum likelihood using two distinct approaches.  In the first approach, we perform numerical maximization of the likelihood function using the \texttt{optim} method, which applies standard gradient-based optimization routines, using the output of the innovations algorithm \cite{BrockwellDavis1988,hannan1982recursive,kailath1968innovations} as the initial value of the optimization. In the second approach, we adopt the viewpoint of numerical algebraic geometry: the likelihood function is differentiated with respect to the moving average parameters, and the resulting system of polynomial equations is solved using \texttt{HomotopyContinuation.jl} \cite{HomotopyContinuation.jl}. This method identifies all stationary points of the likelihood, after which the global maximum is selected by direct evaluation of the likelihood function at each critical point.

Tables~\ref{Tabla:estim_optim} and~\ref{Tabla:estim_Homotopy} summarize the estimation results obtained under both methods.  
Each table reports the true parameter values, the sample means of the estimates, the empirical bias, and the standard deviation across the 500 replications.

\begin{table}[H]
\centering
\begin{tabular}{lcccc}
\toprule
 & True Value & Mean & Bias & Std \\
\midrule
\(a_0\) & 1 & 0.947
& -0.053 & 0.273 \\
\(a_1\) & 0.5 & 0.4831 & -0.016 & 0.460 \\
\(a_2\) & -0.3 & -0.156 & 0.144 & 0.544 \\
\bottomrule
\end{tabular}
\caption{MA($2$) parameter estimation results using \texttt{optim}  ($n = 10$).}
\label{Tabla:estim_optim}
\end{table}

\begin{table}[H]
\centering
\begin{tabular}{lcccc}
\toprule
 & True Value & Mean & Bias & Std \\
\midrule
\(a_0\) & 1 & 0.9828 & -0.072 & 0.130 \\
\(a_1\) & 0.5 & 0.4873 & -0.012 & 0.358 \\
\(a_2\) & -0.3 & -0.131 & 0.169 & 0.325 \\
\bottomrule
\end{tabular}
\caption{MA($2$) parameter estimation results  using \texttt{HomotopyContinuation.jl} ($n = 10$).}
\label{Tabla:estim_Homotopy}
\end{table}

Overall, the results indicate that the \texttt{HomotopyContinuation.jl} \cite{HomotopyContinuation.jl} approach yields slightly smaller bias for parameters $a_0$ and $a_1$ compared to the numerical optimization method, although it produces higher variability in $a_1$ and a modest reduction in the standard deviation of $a_0$. A key advantage of the homotopy continuation approach is that it computes \emph{all} solutions of the underlying polynomial system simultaneously, thereby providing access to the complete set of critical points of the likelihood surface, including the global maximum likelihood (ML) estimate. In contrast, the numerical optimization method \texttt{optim} only converges to a single solution, whose identity is highly sensitive to the choice of initial conditions and may correspond to a local rather than the global maximum.
These findings suggest that, despite its computational complexity, the homotopy-based estimation method provides a valuable benchmark for verifying the presence of local maxima in the likelihood surface and for ensuring that the global maximum likelihood estimator is correctly identified.

\section{\texorpdfstring{MA($3$)}{MA(3)}}
\label{Section:MA(3)}
The MA($3$) model is considerably more challenging than MA($2$): the system gains 
an additional score equation and the number of solutions grows rapidly with $n$. 
As in the MA($2$) case, we use \texttt{HomotopyContinuation.jl}~\cite{HomotopyContinuation.jl} 
to compute the ML degree, with results reported in 
Table~\ref{tab:ml3_degree_combined}.

\begin{table}[H]
    \centering
    \begin{tabular}{|c|cccccc|}
\hline
Sample Size $n$ & 4 & 5 & 6 & 7 & 8 & 9 \\
\hline
Parametric ML degree & 262 & $\geq 1220$ & - & - & - & - \\
Implicit ML degree & 3 & 25 & 65 & 130 & 205 & 300 \\
\hline
\end{tabular}
    \caption{ML degrees of a $\text{MA($3$)}$ process.}
     \label{tab:ml3_degree_combined}
\end{table}

The computation for the parametric ML degree in the case $n=5$  was terminated before 
completion due to computational time constraints; hence, $1220$ represents 
a lower bound on the ML degree in that case.

Unlike the MA($2$) case, Gröbner basis methods are computationally prohibitive 
here, and as reported in~\cite{autocovariance}, no explicit solutions could be 
obtained for general $n$. Nevertheless, for $n = 4$, we present a 
Gröbner basis computation that remains feasible for a particular form of the data and yields the following closed-form result.

\begin{proposition}
Consider a \emph{MA($3$)} process with $n = 4$ sample points of the form
$\bmf{y} = (y_1, y_2, y_3, y_1+y_2-y_3)$.
The two (not necessarily real) solutions $(a_0,a_1,a_2,a_3)$ are given as follows, where we assume $y_2+y_3 \neq 0$ and all denominators appearing are nonzero.
\begin{equation*}
  a_0 \;=\; \frac{-P \pm \sqrt{P^2 - 4QR}}{2Q},
\end{equation*}
where
\begin{small}
\begin{align*}
  Q &= 24(y_1 + 2y_2 + y_3)^2,\\[4pt]
  P &= -4y_1^4 - 26y_1^3y_2 - 60y_1^2y_2^2 - 66y_1y_2^3 - 36y_2^4
       - 10y_1^3y_3 - 42y_1^2y_2y_3 - 78y_1y_2^2y_3 - 78y_2^3y_3\\
    &\quad - 6y_1^2y_3^2 - 36y_1y_2y_3^2 - 78y_2^2y_3^2
           - 8y_1y_3^3 - 40y_2y_3^3 - 8y_3^4,\\[4pt]
  R &= y_1^5y_2 + 5y_1^4y_2^2 + 12y_1^3y_2^3 + 15y_1^2y_2^4 + 9y_1y_2^5
     + y_1^5y_3 + y_1^4y_2y_3 + 2y_1^3y_2^2y_3 + 12y_1^2y_2^3y_3
     + 33y_1y_2^4y_3 \\
    &\quad + 27y_2^5y_3 + 8y_1^3y_2y_3^2 + 6y_1^2y_2^2y_3^2 - 12y_1y_2^3y_3^2
     + 2y_1^3y_3^3 + 8y_1^2y_2y_3^3 + 10y_1y_2^2y_3^3 + 12y_2^3y_3^3\\
    &\quad - y_1^2y_3^4 + 14y_1y_2y_3^4 + 25y_2^2y_3^4
     + 4y_2y_3^5 - 2y_3^6,
\end{align*}
\end{small}
and the remaining unknowns follow by back-substitution:
\begin{small}
\begin{align*}
  a_1 &= \frac{-F_1\,a_0 + G_1}{E_1},\\[6pt]
  a_2 &= \frac{-E_2\,a_1 - F_2\,a_0 + H_2}{D_2},\\[6pt]
  a_3 &= \frac{6(y_1+y_2)\,a_2 + 3(y_2+y_3)\,a_1
              - 6(y_1+y_2)\,a_0 + K}{3(y_2+y_3)},
\end{align*}
\end{small}
where
\begin{small}
\begin{align*}
  E_1 &= 6\!\left(2y_1^3 + 3y_1^2y_2 - 3y_1^2y_3
         + 18y_1y_2y_3 - 9y_2^3 + 21y_2^2y_3 - 8y_3^3\right),\\
  F_1 &= 6y_1^3 - 18y_1^2y_2 - 66y_1y_2^2 - 90y_2^3
        - 36y_1^2y_3 + 48y_1y_2y_3 + 84y_2^2y_3
        - 12y_1y_3^2 - 12y_2y_3^2 - 48y_3^3,\\
  G_1 &= y_1^5 + 2y_1^4y_2 - 3y_1^3y_2^2 - 21y_1^2y_2^3
        - 36y_1y_2^4 - 27y_2^5
        - 3y_1^4y_3 + 4y_1^3y_2y_3 + 30y_1^2y_2^2y_3
        + 48y_1y_2^3y_3 + 9y_2^4y_3\\
      &\quad - y_1^3y_3^2 - 3y_1^2y_2y_3^2 + 33y_1y_2^2y_3^2
        + 27y_2^3y_3^2 - 10y_1^2y_3^3 - 8y_1y_2y_3^3
        - 6y_2^2y_3^3 - 6y_1y_3^4 - 22y_2y_3^4 - 8y_3^5,\\[4pt]
  D_2 &= 2\!\left(11y_1^2 - 27y_2^2 - 22y_1y_3
         - 6y_2y_3 - 16y_3^2\right),\\
  E_2 &= 32y_1^2 + 12y_1y_2 - 72y_2^2 - 52y_1y_3
        + 228y_2y_3 - 52y_3^2,\\
  F_2 &= -6y_1^2 - 44y_1y_2 - 66y_2^2 - 32y_1y_3
        + 200y_2y_3 - 28y_3^2,\\
  H_2 &= y_1^4 - 6y_1^3y_2 + y_1^2y_2^2 + 42y_1y_2^3 + 36y_2^4
        - 10y_1^3y_3 - 12y_1^2y_2y_3 - 12y_1y_2^2y_3 - 42y_2^3y_3\\
      &\quad + 25y_1^2y_3^2 + 32y_1y_2y_3^2 - 61y_2^2y_3^2
        + 18y_1y_3^3 - 14y_2y_3^3 + 2y_3^4,\\[4pt]
  K   &= y_1^3 + 2y_1^2y_2 - y_1^2y_3 - 4y_1y_2y_3
        - y_1y_3^2 + 2y_2y_3^2 + y_3^3.
\end{align*}
\end{small}
\end{proposition}

\begin{proof}
    A Gröbner basis computations reveals the following equations: 
    \begin{small}
    \begin{gather*}
(24y_1^2+96y_1y_2+96y_2^2+48y_1y_3+96y_2y_3+24y_3^2)a_0^2
+(-4y_1^4-26y_1^3y_2-60y_1^2y_2^2 \nonumber \\
-66y_1y_2^3-36y_2^4
-10y_1^3y_3-42y_1^2y_2y_3-78y_1y_2^2y_3-78y_2^3y_3-6y_1^2y_3^2-36y_1y_2y_3^2
-78y_2^2y_3^2
\nonumber\\
-8y_1y_3^3-40y_2y_3^3-8y_3^4)a_0
+y_1^5y_2+5y_1^4y_2^2+12y_1^3y_2^3+15y_1^2y_2^4
+9y_1y_2^5+y_1^5y_3+y_1^4y_2y_3 \nonumber
\\ +2y_1^3y_2^2y_3+12y_1^2y_2^3y_3+33y_1y_2^4y_3
+27y_2^5y_3+8y_1^3y_2y_3^2+6y_1^2y_2^2y_3^2
-12y_1y_2^3y_3^2+2y_1^3y_3^3
\nonumber\\
+8y_1^2y_2y_3^3+10y_1y_2^2y_3^3+12y_2^3y_3^3
-y_1^2y_3^4+14y_1y_2y_3^4+25y_2^2y_3^4+4y_2y_3^5-2y_3^6 =0,
\\[1ex]
(12y_1^3+18y_1^2y_2-54y_2^3-18y_1^2y_3+108y_1y_2y_3
+126y_2^2y_3-48y_3^3)a_1+(6y_1^3-18y_1^2y_2-66y_1y_2^2
\nonumber \\
-90y_2^3
-36y_1^2y_3+48y_1y_2y_3+84y_2^2y_3
-12y_1y_3^2-12y_2y_3^2-48y_3^3)a_0
-y_1^5-2y_1^4y_2 +3y_1^3y_2^2 \nonumber \\
+21y_1^2y_2^3+36y_1y_2^4+27y_2^5+3y_1^4y_3-4y_1^3y_2y_3
-30y_1^2y_2^2y_3-48y_1y_2^3y_3-9y_2^4y_3+y_1^3y_3^2 +3y_1^2y_2y_3^2
\nonumber\\
-33y_1y_2^2y_3^2-27y_2^3y_3^2
+10y_1^2y_3^3+8y_1y_2y_3^3+6y_2^2y_3^3
+6y_1y_3^4+22y_2y_3^4+8y_3^5=0,
\\[1ex]
(22y_1^2-54y_2^2-44y_1y_3-12y_2y_3-32y_3^2)a_2+(32y_1^2+12y_1y_2-72y_2^2-52y_1y_3
\nonumber\\
+228y_2y_3-52y_3^2)a_1+(-6y_1^2-44y_1y_2-66y_2^2-32y_1y_3+200y_2y_3-28y_3^2)a_0+y_1^4
\nonumber\\
-6y_1^3y_2+y_1^2y_2^2+42y_1y_2^3+36y_2^4-10y_1^3y_3-12y_1^2y_2y_3-12y_1y_2^2y_3-42y_2^3y_3+25y_1^2y_3^2
\nonumber\\
+32y_1y_2y_3^2-61y_2^2y_3^2+18y_1y_3^3-14y_2y_3^3+2y_3^4=0,
\\[1ex]
(3y_2+3y_3)a_3+(-6y_1-6y_2)a_2
+(-3y_2-3y_3)a_1+(6y_1+6y_2)a_0-y_1^3-2y_1^2y_2+y_1^2y_3 \nonumber \\+4y_1y_2y_3
+y_1y_3^2-2y_2y_3^2-y_3^3=0. \qedhere
\end{gather*}
\end{small}
\end{proof}
The degenerate cases where one or more of $y_2+y_3, Q, E_1, D_2$ are zero can also be handled explicitly, so that one can give explicit algebraic expressions for the two solutions. We note that this is one less than the ML degree of 3 corresponding to $n=4$, according to Table \ref{tab:ml3_degree_combined}, due to the special nature of the data that was considered.  

\section{The Composite Likelihood Degree}\label{sec:Composite_Likehood}

The composite likelihood (CL) provides a practical alternative to the full likelihood when the latter is analytically intractable or computationally expensive \cite{varin2011overview}. It is constructed by multiplying a collection of lower-dimensional likelihood objects, each associated with a subset of the data, yielding an estimator that retains consistency and asymptotic normality under suitable regularity conditions \cite{lindsay1988composite}.

Several choices of subsets are possible; here we use the \emph{pairwise composite likelihood}, which is built from pairs of correlated observations $(y_k, y_{k+h})$ at lags $h=1,\ldots,q$. In this setting, the pairwise composite log-likelihood is defined as
\begin{equation*}
cl(a) = \sum_{h=1}^{q}\sum_{k=1}^{n-h} \left[-\frac{1}{2(n-h)}\log \det\left(\Sigma_{(k+h)k}\right) -\frac{1}{2(n-h)}(y_{k},y_{k+h})^{\top}\Sigma^{-1}_{k(k+h)}(y_{k},y_{k+h})\right],
\end{equation*}
where $\Sigma_{ij} = \mathrm{Var}[(Y_i,Y_j)]$ is the $2\times 2$ covariance matrix of $(Y_i,Y_j)$ . Defining the theoretical and empirical covariance matrices of order $h \in \{1,\dots,q\}$ as
\begin{equation*}
\Sigma(h) = \mathrm{Var}[(Y_k, Y_{k+h})], \quad 
\hat{\Sigma}(h) = \frac{1}{n-h}\sum_{k=1}^{n-h}(y_{k},y_{k+h})(y_{k},y_{k+h})^{\top},
\end{equation*}
we can express the composite likelihood more compactly as:
\begin{equation*}
cl(a) = \sum_{h=1}^{q} \left[
-\frac{1}{2}\log\det \left (\Sigma(h)\right) 
-\frac{1}{2}\mathrm{tr}\!\left(\Sigma^{-1}(h)\hat{\Sigma}(h)\right)
\right].
\end{equation*}
An important property of the composite likelihood for MA($q$) models is that its number of (generic) complex critical points depends solely on the order $q$ of the process and not on the sample size $n$. This invariance arises because the covariance matrix involved in each term of the CL is always of fixed dimension $2\times2$, regardless of $n$. Consequently, increasing the sample size only refines the estimation of the empirical covariances $\hat{\Sigma}(h)$, without affecting the number of solutions of the composite likelihood equations or the overall algebraic complexity of the model.

\begin{example}
For a first-order moving average process \emph{MA($1$)}, the covariance matrix of order $h=1$ is given by:
\begin{equation*}
\Sigma(1)=
\begin{bmatrix}
a_0^2+a_1^2 & a_0a_1  \\
a_0a_1 & a_0^2+a_1^2
\end{bmatrix},
\end{equation*}
with inverse
\begin{equation*}
\Sigma^{-1}(1)=\frac{1}{(a_0^2+a_1^2)^{2}-(a_0a_1)^2}
\begin{bmatrix}
a_0^2+a_1^2 & -a_0a_1  \\
-a_0a_1 & a_0^2+a_1^2
\end{bmatrix}.
\end{equation*}
Substituting these expressions into the definition of $cl(a)$ yields:
\begin{align*}
cl(a) &=
-\frac{1}{2((a_0^2+a_1^2)^{2}-(a_0a_1)^2)} 
\sum_{k=1}^{n-1}
\left[
\log\!\left((a_0^2+a_1^2)^{2}-(a_0a_1)^2\right)+\right.
\\ & \left.(\hat{\Sigma}_{11}(1)+\hat{\Sigma}_{22}(1))(a_0^2+a_1^2)
-2a_0a_1 \hat{\Sigma}_{12}(1)
\right].
\end{align*}

This expression highlights that the composite likelihood for the \emph{MA($1$)} model has the same algebraic structure as the  likelihood function for $n=2$, but with terms depending on the empirical covariance $\hat{\Sigma}(1)$ rather than directly on the raw observations.

\end{example}

\subsection{ Implicit Composite Likelihood Degree }

The main motivation for formulating the implicit composite likelihood, ie  in terms of the autocovariance parameters 
\(\gamma = (\gamma_0, \gamma_1, \ldots, \gamma_q)\),  is to reduce the algebraic degree of the estimating equations.  
Although the implicit CL degree simplifies the estimating equations compared to the $a_i$-parametrization, the score equations still become computationally intractable as the model order increases
Consider the implicit composite likelihood function.  
For each lag $h$, we define the corresponding covariance and precision matrices as
\begin{equation*}
    \Sigma(h)=
    \begin{bmatrix}
        \gamma_0 & \gamma_h \\
        \gamma_h & \gamma_0
    \end{bmatrix},
    \quad \quad
    \Sigma(h)^{-1}=
    \frac{1}{\gamma_0^2-\gamma_h^2}
    \begin{bmatrix}
        \gamma_0 & -\gamma_h \\
        -\gamma_h & \gamma_0
    \end{bmatrix}.
\end{equation*}
Similarly, the empirical covariance matrix is defined as
\begin{equation*}
    \hat{\Sigma}(h) = \sum_{k=1}^{n-h} (y_k, y_{k+h})(y_k, y_{k+h})^\top.
\end{equation*}
Then, the composite likelihood function can be written as
\begin{equation*}
    cl(\gamma)
    = \sum_{h=1}^{q}
    \Big(
        -\log \det\left(\Sigma(h)\right)
        - \mathrm{tr}\big(\Sigma(h)^{-1}\hat{\Sigma}(h)\big)
    \Big).
\end{equation*}
The score equations are obtained by differentiating $cl(\gamma)$ with respect to each component of $\gamma$.  
For $h \neq 0$, we have
\begin{equation*}
    \frac{\partial}{\partial \gamma_h}\Sigma(h)
    =
    \begin{bmatrix}
        0 & 1 \\
        1 & 0
    \end{bmatrix},
\end{equation*}
and consequently
\begin{align*}
    \frac{\partial cl}{\partial \gamma_h}(\gamma)
    &= 
    \mathrm{tr}\!\left(\Sigma^{-1}(h)\frac{\partial}{\partial \gamma_h}\Sigma(h)\right)
    - \mathrm{tr}\!\left(\Sigma^{-1}(h)\frac{\partial}{\partial \gamma_h}\Sigma(h)\Sigma^{-1}(h)\hat{\Sigma}(h)\right) \\[2mm]
    &= 
    \frac{-2\gamma_h}{\gamma_0^2-\gamma_h^2}
    - \frac{
        2\gamma_0^2 \hat{\Sigma}_{12}(h)
        + 2\gamma_h^2 \hat{\Sigma}_{12}(h)
        - 2(\hat{\Sigma}_{11}(h)+\hat{\Sigma}_{22}(h)) \gamma_0 \gamma_h
    }{(\gamma_0^2-\gamma_h^2)^2}.
\end{align*}
For the variance parameter \(\gamma_0\), the derivative is given by
\begin{align*}
    \frac{\partial cl}{\partial \gamma_0}(\gamma)
    &= \sum_{h=1}^{q}
    \Big[
        \mathrm{tr}\!\left(\Sigma^{-1}(h)\right)
        - \mathrm{tr}\!\left(\Sigma^{-1}(h)\Sigma^{-1}(h)\hat{\Sigma}(h)\right)
    \Big] \\[2mm]
    &= \sum_{h=1}^{q}
    \frac{2\gamma_0}{\gamma_0^2-\gamma_h^2}
    - \frac{
        \gamma_0^2(\hat{\Sigma}_{11}(h)+\hat{\Sigma}_{22}(h))
        - 4\gamma_0\gamma_h \hat{\Sigma}_{12}(h)
        + \gamma_h^2(\hat{\Sigma}_{11}(h)+\hat{\Sigma}_{22}(h))
    }{(\gamma_0^2-\gamma_h^2)^2}.
\end{align*}
We compute the implicit and parametric composite likelihood degree (CL degree) with \cite{HomotopyContinuation.jl} for diferents models ($q$), the results obtained are summarized in Table~\ref{tab:cl_degree}. Entries marked with -- correspond to cases that could not be computed within a reasonable time limit (1 day).
\begin{table}[H]
    \centering
    \begin{tabular}{|c|c|c|c|c|c|c|c|}
\hline
MA$(q)$ & 1 & 2 & 3 & 4 & 5 & 6 & 7 \\
\hline
Parametric CL degree & 8 & 122 & $\geq$ 3698 & -- & -- & -- & -- \\
\hline
Implicit CL degree & 1 & 7 & 25 & 79 & 234 & 671 & 2159 \\
\hline
\end{tabular}
    \caption{Composite likelihood degrees for different MA($q$) process. }
    \label{tab:cl_degree}
\end{table}
Note that for $q = 1$, the composite likelihood coincides with the likelihood
function based on $n = 2$ observations, thus yielding the same results as those reported in Table~\ref{tab:Ma1}. In contrast, for $q = 2$ we already
observe degrees that are lower than those in Table~\ref{tab:MA(2)_gamma_res} for $n\geq 5$.

\begin{example}
For \emph{MA($2$)} the solutions of the parametric CL degree we obtain $122$ and $7$ for the parametric and implicit, respectively. The additional solutions of the parametric correspond to the non-invertible case (i.e. $\det(J_g^{\top})=0$). In
particular we obtain numerically that the solutions are organized in five groups:
\begin{itemize}
    \item $2 \cdot 12$ solutions with $a_0 = a_2$ and $a_1 \neq 0$.
    \item $4 \cdot 5$ solutions with $a_0 + a_1 + a_2 = 0$ and $a_1 \neq 0$.
    \item $4 \cdot 5$ solutions with $a_0 - a_1 + a_2 = 0$ and $a_1 \neq 0$.
    \item $2 \cdot 1$ solutions with $a_0 + a_2 = 0$ and $a_1 = 0$.
    \item $8 \cdot 7$ solutions for which the model is invertible.
\end{itemize}
which are the same cases that we obtain with the parametric ML degree.
\end{example}

\subsection{Simulation Study of Parameter Estimation}\label{sec:Estimation:Composite}
To further evaluate the performance of the composite likelihood approach, we conducted a simulation study analogous to the one performed for the full likelihood (Section~\ref{section:Estimation_MA_process}), now with $n = 500$ observations per sample and using the composite likelihood function in place of the full likelihood.
Tables~\ref{tab:est_cl_optim} and~\ref{tab:est_cl_homotopy} summarize the results obtained with \texttt{optim} and \texttt{HomotopyContinuation.jl} \cite{HomotopyContinuation.jl}, respectively.

\begin{table}[H]
\centering
\begin{tabular}{lcccc}
\toprule
 & True Value & Mean & Bias & Std \\
\midrule
\(a_0\) & 1   & 1.0943 & 0.094 & 0.181 \\
\(a_1\) & 0.5 & 0.6361 & 0.136 & 0.132 \\
\(a_2\) & -0.3 & -0.3453 & -0.045 & 0.170 \\
\bottomrule
\end{tabular}
\caption{MA($2$) parameter estimation results using \texttt{optim} ($n=500$).}
\label{tab:est_cl_optim}
\end{table}

\begin{table}[H]
\centering
\begin{tabular}{lcccc}
\toprule
 & True Value & Mean & Bias & Std \\
\midrule
\(a_0\) & 1   & 1.1450 & 0.145 & 0.055 \\
\(a_1\) & 0.5 & 0.5855 & 0.085 & 0.101 \\
\(a_2\) & -0.3 & -0.3245 & -0.0245 & 0.081 \\
\bottomrule
\end{tabular}
\caption{MA($2$) parameter estimation results using \texttt{HomotopyContinuation.jl} ($n=500$).}
\label{tab:est_cl_homotopy}
\end{table}

Under the composite likelihood framework, \texttt{optim} yields estimators with consistently larger variance across all parameters compared to \texttt{HomotopyContinuation.jl}; although it attains a smaller bias for $a_0$, it performs worse in both bias and variability for $a_1$ and $a_2$. Consistent with the behavior observed for the full likelihood, this reflects the sensitivity of \texttt{optim} to the choice of initial values, whereas \texttt{HomotopyContinuation.jl} remains stable and yields solutions close to the true parameters regardless of the starting point.

\section{Autoregressive Processes}\label{sec:autoregresive}

We now turn to the autoregressive process of order $p$ (AR($p$)), previously introduced in Definition~\ref{Def:Autoregressive}. According to Equation~\eqref{Eq:autoregresive_process}, an autoregressive process can be written as
\begin{equation*}
    X_t = \sum_{i=1}^{p} \phi_i X_{t-i} + Z_t,
\end{equation*}
where $\phi_i\in \mathbb{R}$, $\sigma \in \mathbb{R}$ and $Z_t \sim \mathcal{N}(0, \sigma^2)$ are i.i.d. Gaussian.  Unlike a moving average (MA) process, which is always stationary but not necessarily invertible, an autoregressive (AR) process is always invertible but may fail to be stationary. Stationarity of the AR($p$) model requires that the roots of the characteristic polynomial
\begin{equation*}
    \Phi(z) = 1 - \sum_{i=1}^{p} \phi_i z^i
\end{equation*}
lie outside the unit circle in the complex plane. When this condition holds, the process admits an infinite-order moving average representation and possesses well-defined second-order moments. 

In the AR($p$) case, the autocovariance function $\gamma(h)$ does not have compact support, which prevents working directly with a finite covariance matrix $\Sigma$ as in the moving average case. Moreover, $\gamma(h)$ typically lacks a closed-form expression, being determined recursively through the Yule–Walker equations. 

We consider now $t\in \mathbb{N}^{+}$ and note that for $t > p$, the process can be expressed as
\begin{equation*}
    X_t - \sum_{i=1}^{p} \phi_i X_{t-i} = Z_t,
\end{equation*}
while for the initial $p$ observations ($1 \leq t \leq p$), we have  
\begin{equation*}
\begin{pmatrix}
X_1\\
X_2-\phi X_1\\
\vdots \\
X_{p-1} - \sum_{i=1}^{p-2} \phi_i X_{p-1-i} \\
X_p - \sum_{i=1}^{p-1} \phi_i X_{p-i} \\

\end{pmatrix}
\sim \mathcal{N}_p(0, \sigma^2 \Sigma'),
\end{equation*}
where $\Sigma'$ is an autocovariance matrix whose explicit form is unknown. Since $\sigma^2$ is a multiplicative scalar factor of $\Sigma'$, for simplicity we set $\sigma^2=1$. 
Define 
\begin{equation*}
    D_n = 
    \begin{pmatrix}
        \Sigma' & 0 \\
        0 & I_p
    \end{pmatrix}.
\end{equation*}
Let $X = (X_1, \ldots, X_n)^{\top}$ and define the $n \times n$ upper-triangular matrix $\Phi$ as 
\begin{equation*}
(U_{\phi})_{t,s} =
\begin{cases}
1, & t=s, \\
-\phi_{s-t}, & 1 \le s-t \le p, \\
0, & s-t > p \text{ or } t>s.
\end{cases}
\end{equation*}
Then the first $p$ entries of $U_{\phi} X$ represents the first $p$ initial observations and the remaining $n-p$ entries corresponding to the AR equations.
Defining \(Z = (Z_1, \ldots, Z_n)^{\top}\), we can express the model compactly as:
\begin{equation*}
    X U_{\phi} = Z \sim \mathcal{N}_n(0, D_n),  \Rightarrow \;     X \sim \mathcal{N}_n\left(0, (U_{\phi}^{-1})^{\top} D_n U_{\phi}^{-1}\right),
\end{equation*}
so that the covariance matrix is given by $\Sigma = (U_{\phi}^{-1})^{\top} D_n U_{\phi}^{-1}$.

Since $\Sigma$ is unknown in closed form, it is convenient to work instead with the precision matrix $Q=\Sigma^{-1}$. The following result of \cite{siddiqui1958inversion}  characterizes the structure of $Q$.
\begin{proposition}{\cite{siddiqui1958inversion}}
Let $X_t = \sum_{i=1}^{p} \phi_i X_{t-i} + \epsilon_t$ be a stationary \emph{AR($p$)} process, and set $\phi_0=-1$. Then, the precision matrix $Q$ is given by:
\begin{equation*}
Q_{ij} = 
\begin{cases}
\displaystyle \sum_{k = 0}^{\min(p - i - j,\, j - 1,\, n - j)} \phi_k \, \phi_{k + |j - i|}, & \text{if } i < j, \\[5pt]
Q_{ji}, & \text{if } i > j.
\end{cases}
\end{equation*}
with $\phi_0=-1$.
\end{proposition}

This formula was proven in \cite{siddiqui1958inversion}; here we provide a simpler argument: 

\begin{proof}
Since $U_\phi$ is upper triangular with ones on the diagonal and $D_n^{-1}$ is diagonal, $Q = U_\phi D_n^{-1} U_\phi^\top$ is symmetric, and the banded structure of $U_\phi$ (bandwidth $p$) implies $Q$ is banded with bandwidth $p$.

For $i<j$, since $(U_\phi)_{i,u}\neq 0$ only for $u\in\{i,\dots,i+p\}$ and $(U_\phi)_{j,u}\neq 0$ only for $u\in\{j,\dots,j+p\}$, the product
\begin{equation*}
Q_{ij} = \sum_{u=1}^{n} (U_\phi)_{i,u}\,(D_n^{-1})_{u,u}\,(U_\phi)_{j,u}
\end{equation*}
only involves $u=j+k$ with $k=0,\dots,\min(p-|i-j|,\,n-j)$. Writing $(U_\phi)_{j,j+k}=-\phi_k$ and $(U_\phi)_{i,j+k}=-\phi_{k+|i-j|}$ (using $\phi_0=-1$ to unify $k=0$ with the diagonal entries), each term becomes $\phi_k\,\phi_{k+|i-j|}\,(D_n^{-1})_{j+k,j+k}$, so that
\begin{equation*}
Q_{ij} = \sum_{k=0}^{\min(p-|i-j|,\,n-j)} \phi_k\,\phi_{k+|i-j|}\,(D_n^{-1})_{j+k,j+k}.
\end{equation*}
For $j>n-p$, all indices $j+k$ exceed $p$, so $(D_n^{-1})_{j+k,j+k}=1$ and the sum reduces exactly to $\sum_{k=0}^{\min(p-|i-j|,\,n-j)} \phi_k\phi_{k+|i-j|}$, determining $Q$ on the last $p$ columns without reference to the boundary terms of $D_n$.

The remaining entries follow by persymmetry: since $X$ is stationary, $\Sigma$ is Toeplitz, hence persymmetric, and therefore so is its inverse $Q$, i.e.\ $Q_{ij}=Q_{n+1-j,\,n+1-i}$. Reflecting the formula above onto the first $p$ columns replaces $n-j$ by $j-1$, yielding
\begin{equation*}
Q_{ij} = \sum_{k=0}^{\min(p-|i-j|,\,j-1,\,n-j)} \phi_k\,\phi_{k+|i-j|}, \qquad i<j,
\end{equation*}
and $Q_{ij}=Q_{ji}$ for $i>j$, completing the proof.
\end{proof}

\begin{example}

In the \emph{AR($2$)} case with $n=8$, the precision matrix $Q$ takes the form:
\begin{equation*} \begin{adjustbox}{max width=\textwidth} \small $\begin{bmatrix} 1 & -\phi_1 & -\phi_2 & 0 & 0 & 0 & 0 & 0 \\ -\phi_1 & 1+\phi_1^2 & -\phi_1+\phi_1\phi_2 & -\phi_2 & 0 & 0 & 0 & 0 \\ -\phi_2 & -\phi_1+\phi_1\phi_2 & 1+\phi_1^2+\phi_2^2 & -\phi_1+\phi_1\phi_2 & -\phi_2 & 0 & 0 & 0 \\ 0 & -\phi_2 & -\phi_1+\phi_1\phi_2 & 1+\phi_1^2+\phi_2^2 & -\phi_1+\phi_1\phi_2 & -\phi_2 & 0 & 0 \\ 0 & 0 & -\phi_2 & -\phi_1+\phi_1\phi_2 & 1+\phi_1^2+\phi_2^2 & -\phi_1+\phi_1\phi_2 & -\phi_2 & 0 \\ 0 & 0 & 0 & -\phi_2 & -\phi_1+\phi_1\phi_2 & 1+\phi_1^2+\phi_2^2 & -\phi_1+\phi_1\phi_2 & -\phi_2 \\ 0 & 0 & 0 & 0 & -\phi_2 & -\phi_1+\phi_1\phi_2 & 1+\phi_1^2 & -\phi_1 \\ 0 & 0 & 0 & 0 & 0 & -\phi_2 & -\phi_1 & 1 \end{bmatrix}$ \end{adjustbox} \end{equation*}
\end{example}

Similarly, for the AR($3$) case, the precision matrix $Q$ extends naturally by including one upper and lower additional band of nonzero entries corresponding to the parameter $\phi_3$.

\subsection{The \texorpdfstring{AR($1$)}{AR(1)} Model}

We now consider the first-order autoregressive process, denoted AR($1$), defined as:
\begin{equation*}
    X_t = \phi X_{t-1} + Z_t, \quad Z_t \sim \mathcal{N}(0, \sigma^2), \quad t = 1, \dots, n.
\end{equation*}
Provided that $|\phi| < 1$, the process is stationary with variance
$
    \operatorname{Var}(X_t) = \dfrac{\sigma^2}{1 - \phi^2},
$ and the log-likelihood function (up to an additive constant) is
\begin{equation*}
    \ell(\phi,\sigma^2) \propto -\log \det \Sigma - \mathbf{x}^{\top} \Sigma^{-1} \mathbf{x},
\end{equation*}
where $\mathbf{x} = (x_1, \dots, x_n)^{\top}$ and $\Sigma$ denotes the covariance matrix of $(X_1,\dots, X_n)^{\top}$.  The ML solutions are given by the following theorem

\begin{theorem}
The ML degree for the \emph{AR($1$)} with parameters $(\phi,\sigma^2)$  is $3$ for any sample size $n>1$. 
Defining the constants
\begin{align*}
b_3 = -(n-1) \sum_{i=2}^{n-1} x_i^2, \quad  b_2 = (n-2) \sum_{i=1}^{n-1} x_i x_{i+1}, \\ b_1 = (n+1) \sum_{i=2}^{n-1} x_i^2 + (x_1^2 + x_n^2), \quad  b_0 = -n\sum_{i=1}^{n-1} x_i x_{i+1},
\end{align*}
 the three solutions of $\phi$ are given by 
    \begin{equation*}
\phi_k = -\frac{b_2}{3b_3}
+ \frac{1}{3b_3}\left(
\omega_k \sqrt[3]{U} + \omega_k^2 \sqrt[3]{V}
\right),
\qquad k = 1,2,3
\end{equation*}
where
\begin{align*}
p &= \frac{3b_3 b_1 - b_2^2}{3b_3^2}, 
\qquad
q = \frac{2b_2^3 - 9b_3 b_2 b_1 + 27b_3^2 b_0}{27b_3^3}, \\[1mm]
\Delta &= -\left(\frac{q}{2}\right)^2 - \left(\frac{p}{3}\right)^3, \qquad
U = -\frac{q}{2} + \sqrt{\Delta}, 
\qquad
V = -\frac{q}{2} - \sqrt{\Delta}, \\[1mm]
\omega_1 &= 1, \quad 
\omega_2 = e^{2\pi i/3} = -\tfrac{1}{2} + i\tfrac{\sqrt{3}}{2}, \quad 
\omega_3 = e^{4\pi i/3} = -\tfrac{1}{2} - i\tfrac{\sqrt{3}}{2}.
\end{align*}
and $\sigma^2$ is then determined by:
\begin{equation*}
    \sigma^2=\frac{1}{n}\mathbf{x}^{\top}R^{-1}\mathbf{x}
\end{equation*}
with $R=R(\phi)$ the autocorrelation matrix of $(X_1,\dots,X_n)^{\top}$.
\end{theorem}

\begin{proof}

Note that $\Sigma = \sigma^2 R$ and solving the score equation with respect to $\sigma^2$ gives $\sigma^2 = \frac{1}{n} \mathbf{x}^{\top} R^{-1} \mathbf{x}$, which substituted back into the log-likelihood yields (up to additive constant):
\begin{equation*}
    \ell(\phi) \propto -\log \det R - n \log\left( \mathbf{x}^{\top} R^{-1} \mathbf{x} \right).
\end{equation*}
Differentiating with respect to $\phi$, the score equation becomes
\begin{align}
    \frac{\partial \ell}{\partial \phi}(\phi)
    &= \operatorname{tr}\left(\frac{\partial R^{-1}}{\partial \phi} R \right)
    - n \frac{\mathbf{x}^{\top} \frac{\partial R^{-1}}{\partial \phi} \mathbf{x}}{\mathbf{x}^{\top} R^{-1} \mathbf{x}} = 0. \label{Eq:ScoreAR}
\end{align}
The correlation matrix $R$ has a compact form $ R_{ij} = \frac{\phi^{|i-j|}}{1 - \phi^2}$, with $i,j = 1, \dots, n,
$, namely:
\begin{equation*}
R = \frac{1}{1 - \phi^2}
\begin{pmatrix}
1 & \phi & \phi^2 & \cdots & \phi^{n-1} \\
\phi & 1 & \phi & \cdots & \phi^{n-2} \\
\phi^2 & \phi & 1 & \cdots & \phi^{n-3} \\
\vdots & \vdots & \vdots & \ddots & \vdots \\
\phi^{n-1} & \phi^{n-2} & \phi^{n-3} & \cdots & 1
\end{pmatrix}, \quad R^{-1} =
\begin{pmatrix}
1 & -\phi & 0 & \cdots & 0 \\
-\phi & 1 + \phi^2 & -\phi & \cdots & 0 \\
0 & -\phi & 1 + \phi^2 & \cdots & 0 \\
\vdots & \vdots & \vdots & \ddots & \vdots \\
0 & 0 & 0 & \cdots & 1
\end{pmatrix}.
\end{equation*}
Differentiating $R^{-1}$ with respect to $\phi$ and substituting into \eqref{Eq:ScoreAR}  yields:
\begin{equation*}
    -2\phi X^{\top} R^{-1}X - (1-\phi^2)nX^{\top}\frac{\partial R^{-1}}{\partial \phi}X = 0.
\end{equation*}

Expanding the quadratic forms, we obtain:
\begin{align*}
& -2\phi \Bigg(
\sum_{i=2}^{n-1} (1+\phi^2) X_i^2
- 2\phi \sum_{i=1}^{n-1} X_i X_{i+1}
+ X_1^2 + X_n^2
\Bigg) \\
& - n(1-\phi^2) \Bigg(
2\phi \sum_{i=2}^{n-1} X_i^2
- 2 \sum_{i=1}^{n-1} X_i X_{i+1}
\Bigg) = 0.
\end{align*}
Rearranging terms, with the defined $b_i$ coefficients, this leads to the cubic equation: 
\begin{equation*}
    b_3 \phi^3 + b_2 \phi^2 + b_1 \phi + b_0 = 0.
\end{equation*}
The expressions for the three solutions are given by Cardano's formula.
\end{proof}

\subsection{\texorpdfstring{AR($p$)}{AR(p)} Model}
Unlike moving average models, the algebraic complexity of an autoregressive process depends on the order $p$ rather than the sample size $n$. In particular, the ML degree is intrinsic to the model structure and remains constant regardless of the number of observations.  

We computed the ML degrees for several autoregressive processes of increasing order. The results, obtained using \cite{HomotopyContinuation.jl} are summarized in Table~\ref{tab:ml_degree_ar}.

\begin{table}[H]
    \centering
    \begin{tabular}{|c | c c c c c c|}
        \hline
        AR$(p)$ & 1 & 2 & 3 & 4 & 5 & 6 \\
        \hline
        ML degree & 3 & 7 & 17 & 41 & 190 & 239 \\
        \hline
    \end{tabular}
    \caption{ML degree for different AR($p$) models.}
    \label{tab:ml_degree_ar}
\end{table}

\subsection{Simulation Study of Parameter Estimation}\label{Estimation:AR}
In the statistical literature \cite{brockwell2009time,hamilton1994time,box2015time}, the most common approach for estimating autoregressive models assumes that the first observations, typically $X_1$ and $X_2$, are known and fixed. Under this assumption, the likelihood simplifies substantially, allowing for straightforward analytical or numerical estimation of the parameters.  
In contrast, our formulation treats all observations as random and derives the full likelihood without conditioning on the initial values, leading to a more general and exact estimation framework. Here, the precision matrix explicitly accounts for the boundary terms in the process definition.  
We evaluate this approach for the AR($1$) model using the same two estimation strategies as in Section~\ref{section:Estimation_MA_process}: \texttt{optim} and \texttt{HomotopyContinuation.jl} \cite{HomotopyContinuation.jl}. The resulting parameter estimates are summarized in Tables~\ref{tab:optim_estimation} and~\ref{tab:homotopy_estimation}.

\begin{table}[H]
\centering
\begin{tabular}{lcccc}
\toprule
 & True Value & Mean & Bias & Std Dev \\
\midrule
\(\phi\) & 0.5 & 0.3582 & -0.1418 & 0.0753 \\
\(\sigma^2\) & 1.0 & 0.8800 & -0.1201 & 0.2000 \\
\bottomrule
\end{tabular}
\caption{AR($1$) parameter estimation results using \texttt{optim} ($n = 15$).}
\label{tab:optim_estimation}
\end{table}

\begin{table}[H]
\centering
\begin{tabular}{lcccc}
\toprule
 & True Value & Mean & Bias & Std Dev \\
\midrule
\(\phi\) & 0.5 & 0.5604 & 0.0604 & 0.2703 \\
\(\sigma^2\) & 1.0 & 0.9265 & -0.0735 & 0.3939 \\
\bottomrule
\end{tabular}
\caption{AR($1$) parameter estimation results using \texttt{HomotopyContinuation.jl} ($n=15$).}
\label{tab:homotopy_estimation}
\end{table}

The results show a trade-off similar to that observed for the moving average models: the \texttt{HomotopyContinuation.jl} approach yields smaller bias for both $\phi$ and $\sigma^2$ compared to \texttt{optim}, but at the cost of substantially higher variability. These findings reinforce the conclusion that, despite its higher computational cost, the algebraic approach provides a more reliable benchmark for assessing the true maximum likelihood estimator.

\section*{Acknowledgments}

The authors are grateful to Benjamin Hollering and Bernd Sturmfels for helpful discussions. The authors acknowledge funding by the Deutsche Forschungsgemeinschaft (DFG, German Research Foundation) – CRC/TRR 388 ``Rough Analysis, Stochastic Dynamics and Related Fields'' – Project B01,  ID 516748464.

\bibliographystyle{alpha}
\bibliography{ref.bib}

\vspace{0.25cm}

\noindent{\bf Authors' addresses:}
\smallskip
\small 

\noindent Carlos Am\'{e}ndola,
Technische Universit\"at Berlin
\hfill {\tt amendola@math.tu-berlin.de}

\noindent Gabriel Riffo,
Technische Universit\"at Berlin
\hfill {\tt riffo@math.tu-berlin.de}

\end{document}